\documentclass[12pt,reqno]{amsart}
\usepackage{geometry}                
\usepackage{a4wide}                   
\usepackage{amssymb}
\usepackage{mathtools}
\mathtoolsset{showonlyrefs}

\usepackage{amsmath}
\usepackage{amsthm}
\usepackage{pdfsync}
\usepackage{hyperref}

\usepackage{latexsym}
\usepackage{indentfirst}

\usepackage{amsthm}
\usepackage{bm}
\usepackage{accents}
\usepackage{enumitem}
\usepackage{relsize}
\usepackage{datetime}
\usepackage[toc,page]{appendix}

\usepackage{bigints}
\usepackage{bbm}
\usepackage{sidecap}
\usepackage{caption}
\usepackage{xfrac}
\usepackage{kantlipsum}
\allowdisplaybreaks
\usepackage[numbered]{bookmark}
\usepackage{hyperref}
\usepackage{url}
\hypersetup{pdftex,colorlinks=true,linkcolor=blue,urlcolor=blue,citecolor=blue}
\usepackage{hypcap}

\numberwithin{equation}{section}

\renewcommand*{\d}{\, \mathrm{d}}
\allowdisplaybreaks
\def\mathclap#1{\text{\hbox to 0pt{\hss$\mathsurround=0pt#1$\hss}}}

\newtheorem{thm}{Theorem}[section]

\newtheorem{lemma}[thm]{Lemma}
\newtheorem{theorem}[thm]{Theorem}
\newtheorem{remark}[thm]{Remark}

\def\rr{\mathbb{R}}

\def\cn{\mathbb{C}}

\def\bu{{\bf{u}}}

\def\bv{{\bf{V}}}

\def\bsv{{\bf{v}}}
\def\ue{u^e}
\def\ve{v^e}
\def\cn{\mathbb{C}}

\newcommand{\R}{\mathbb{R}}
\newcommand{\LL}{L}

\title[]{ Hardy uniqueness principle for the linear Schr\"odinger equation on quantum regular trees }
\author[A. F. Bertolin, A.Grecu, L.I. Ignat ]{Aingeru Fern\'andez Bertolin \and Andreea Grecu  \and Liviu I. Ignat
}

\address{A. Fern\'andez-Bertolin
\hfill\break\indent Universidad del Pa\'is Vasco / Euskal Herriko Unibertsitatea\\
\hfill\break\indent  Dpto. de Matem\'aticas \\ Apto. 644 \\ Bilbao \\ Spain  
}
  \email{{\tt
aingeru.fernandez@ehu.eus}}

\address{A. Grecu
\hfill\break\indent University of Bucharest \\ 14 Academiei Street \\010014 Bucharest, Romania\\  
\hfill\break\indent Institute of Mathematics ``Simion Stoilow'' of the Romanian Academy\\
 21 Calea Grivitei Street \\010702 Bucharest \\ Romania 
}
 \email{{\tt
andreea.grecu@my.fmi.unibuc.ro}}

\address{L. I. Ignat
\hfill\break\indent Institute of Mathematics ``Simion Stoilow'' of the Romanian Academy,
Centre Francophone en Math\'{e}matique
\\21 Calea Grivitei Street \\010702 Bucharest \\ Romania\\
 \hfill\break\indent  ICUB,   The Research Institute of the University of Bucharest, University of Bucharest\\
36-46 Bd. M. Kogalniceanu,   050107, Bucharest, Romania\\
}
 \email{{\tt
liviu.ignat@gmail.com}   {\it Web page: }{\tt
http://www.imar.ro/\~\,lignat}}

\begin{document}


\begin{abstract}
In this paper we consider the linear Schr\"odinger equation (LSE) on a regular tree with the last generation of edges of infinite length and analyze some unique continuation properties. The first part of the paper deals with the LSE on the real line with a piece-wise constant coefficient and uses this result in the context of regular trees. The second part treats the case of a LSE with a real potential in the framework of a star-shaped graph.

\end{abstract}

\maketitle

\section{Introduction}
For any function $f\in L^2(\rr) $ we consider its Fourier transform
\[
\hat{f}(\xi)=\frac 1{\sqrt{2\pi}}\int_{\rr} \mathrm{e}^{-\mathrm{i} x\xi} f(x)\d x,\ \xi\in \rr.
\] 
With the above definition in mind, the well known Hardy's uniqueness principle (HUP) \cite[Theorem 2]{hardyup}, asserts that if $f$ and $\hat{f}$ are both $O(\mathrm{e}^{-\frac{1}{2}x^2})$, then $f=g=A \mathrm{e}^{-\frac{1}{2}x^2}$, with $A$ a constant, and if one is $o(\mathrm{e}^{-\frac{1}{2}x^2})$, then both are identically zero. As a consequence, if
\begin{equation*}
f(x)=O(\mathrm{e}^{-\alpha x^2}) \text{ and } \hat{f}(\xi)=O(\mathrm{e}^{-\beta \xi^2})
\end{equation*}
 with $\alpha, \beta>0$ such that $\alpha\beta >1/4$, then $f\equiv 0$. This result is sharp, in the sense that if $
\alpha\beta=1/4$  then $f$ is a multiple of $\mathrm{e}^{-\alpha x^2}$. Morgan \cite{morganup} extends this result to any conjugate exponents $p$ and $p'=\frac{p}{p-1}$ with $p > 2$. More precisely, if 
\begin{equation*}
f(x)=O(\mathrm{e}^{-\alpha x^p}) \text{ and } \hat{f}(\xi)=O(\mathrm{e}^{-\beta \xi^{p'}}) \quad \text{ as } |x|, |\xi| \to +\infty,
\end{equation*}
 with $\alpha, \beta>0$ such that $\alpha^{1/p} \beta^{1/p'} > \frac{1}{p^{1/p} p'^{1/p'}} |\cos(\frac{\pi p'}{2})|^{1/p'}$, then $f\equiv 0$. This result is also sharp. One-sided versions of these results are obtained by Nazarov \cite[Theorem 2.3]{nazarov}: for $p \in [2,\infty]$ if 
\begin{equation*}
f(x)=O(\mathrm{e}^{-\alpha x^p}) \text{ and } \hat{f}(\xi)=O(\mathrm{e}^{-\beta \xi^{p'}}) \quad \text{ as } x, \xi \to -\infty \text{ or } +\infty,
\end{equation*}
 with $\alpha, \beta>0$ such that $\alpha^{1/p} \beta^{1/p'} > \frac{1}{p^{1/p} p'^{1/p'}} \sin(\frac{\pi}{p'})$, then $f\equiv 0$. The exponents in this case are also the best possible. 

 Cowling and Price \cite{cowlingprice} extend to $L^p, L^{q}$ versions: if $1 \leq p,q \leq \infty$ with at least one of them finite and
 \[
 \| \mathrm{e}^{\alpha x^2} f\|_{L^p(\rr)}+ \| \mathrm{e}^{\beta \xi^2} \hat f\|_{L^{q}(\rr)}<\infty,
 \]
with $\alpha, \beta>0$, such that $\alpha\beta > 1/4$, then $f\equiv 0$. The proofs of the above results use complex analysis techniques, and similar results in terms of the unique solution in $C(\mathbb{R},L^2(\mathbb{R}))$ of the linear Schr\"odinger equation 
\begin{align}\label{sch.real.line}
\left\{
\begin{array}{ll}
\mathrm{i} u _t(t,x)+\Delta  u(t,x)=0,& x\in \rr, t\neq 0 ,\\
u(0)=u_0,&   x\in \rr
\end{array}
\right.
\end{align}
can be obtained, see for e.g. \cite{cowlingekpv2010}. Using the Fourier transform  the solution $u$ of the above sistem can be written as 
 \[
u(t,x)=\dfrac{1}{\sqrt{2 \mathrm{i} t}} \mathrm{e}^{\frac{\mathrm{i}|x|^2}{4t}} \Big( \mathrm{e}^{\frac{\mathrm{i}|\cdot|^2}{4t}}u_0\Big)\, \widehat{}\,(\frac {x}{2t}).
 \]
 This representation and the above property of the Fourier transform show that the unique solution of system \eqref{sch.real.line} satisfying $u(0,x)=O(\mathrm{e}^{-\alpha x^2})$, $u(T,x)=O(\mathrm{e}^{-\beta x^2})$ as $|x| \to \infty$, with 
\[
\alpha \beta > \frac{1}{16 T^2},
\]
vanishes identically. $L^p$-versions of these results hold also under the same assumption.
For convenience, in the following we will consider the case $T=1$. 
 
 In this paper we obtain similar results for the Schr\"odinger equation on trees. 
Let us consider the Schr\"odinger equation on   a  tree   $\Gamma$: 
\begin{align}\label{eq.tree}
\left\{
\begin{array}{ll}
\mathrm{i}\bu _t(t,x)+\Delta_\Gamma \bu(t,x)=0,& x\in \Gamma, t\neq 0 ,\\
\bu(0)=\bu_0,&   x\in \Gamma,
\end{array}
\right.
\end{align}
where with $\Delta_\Gamma$ is the Laplace operator on $\Gamma$ with  the Kirchhoff coupling condition at the vertices (see section \ref{notations} for a precise definition).

Our main result concerning the HUP for the above system is obtained in the context of regular trees. These are particular cases of trees having the property that all the edges of the same generation have the same length and all the vertices of the same generation have equal number of children. In the following we write $f\lesssim g$ if there exists a positive constant $C$, depending on $f$ and $g$, such that $f\leq C g$.

\begin{theorem} \label{free-sch}
	Let $\Gamma$ be a regular tree and $\alpha$ and $\beta$ such that $\alpha 
	\beta>1/16$. Any solution $\bu\in C(\rr,L^2(\Gamma)) $ of problem \eqref{eq.tree} that satisfies
	\begin{align}
\label{decay}
 |\bu(0,x)|\lesssim  \mathrm{e}^{-\alpha x^2}, \ |\bu(1,x)|\lesssim \mathrm{e}^{-\beta x^2}, \quad \forall\ x\in \Gamma
\end{align}
 vanishes identically. 
\end{theorem}

Using the arguments in \cite{ignatsiam} one can reduce the properties of the solutions of the LSE on a regular tree to the analysis of the LSE involving a piecewise constant coefficient $\sigma$. 
Theorem \ref{free-sch} is a consequence of the following result for the linear Schr\"odinger equation with a piecewise constant coefficient $\sigma : \mathbb{R} \to \mathbb{R}$ taking a finite number of positive values:
	\begin{equation}\label{eq.sigma}
\left\{
\begin{array}{ll}
\mathrm{i} u _t(t,x)+ \partial_x (\sigma \partial_x u)(t,x)=0,& x\in \rr, t\neq 0 ,\\ 
u(0,x)=u_0(x),&   x\in \rr.
\end{array}
\right.
\end{equation} 
For a precise statement we introduce the two values of $\sigma$ at $\pm \infty$:
\[
\sigma_ {-}=\lim _{x\rightarrow -\infty} \sigma(x),\ \sigma_ {+}=\lim _{x\rightarrow +\infty} \sigma(x).
\]

\begin{theorem} \label{piecewise theorem}
	\label{piecewise}
Any solution $u\in C(\rr,L^2(\rr))$ of system \eqref{eq.sigma} satisfying for some positive $\alpha, \beta$ one of the following
assumptions 
\begin{equation} \label{2 moments decay}
\begin{aligned}
&(i) \ |u(0,x)|\lesssim \mathrm{e}^{-\alpha x^2}, \ |u(1,x)|\lesssim \mathrm{e}^{-\beta x^2}, \quad \text{ as } x \to -\infty, \quad  \alpha\beta > \dfrac{1}{16\sigma _{-}^2},\\
&(ii) \ |u(0,x)|\lesssim \mathrm{e}^{-\alpha x^2}, \ |u(1,x)|\lesssim \mathrm{e}^{-\beta x^2}, \quad \text{ as } x \to +\infty,  \quad \alpha\beta > \dfrac{1}{16 \sigma_{+}^2 }, \\
&(iii) \ |u(0,x)|\lesssim \mathrm{e}^{-\alpha x^2}, \ |u(1,x)|\lesssim \mathrm{e}^{-\beta x^2}, \quad \text{ as } |x| \to \infty,  \quad \alpha\beta > \dfrac{1}{16 \max \{\sigma _{-}^2, \sigma_{+}^2 \}},
 \end{aligned}
\end{equation}
vanishes identically. Moreover, in the case when $\sigma$ is a two-step piecewise constant function (i.e, it takes only two values), these exponents are sharp.
\end{theorem}
In spirit of \cite{cowlingprice}, $L^2$-versions may be obtained, but it its beyond the scope of this paper. 

In the case of Schr\"odinger equation with a potential $\bv=(V_1,\dots, V_N):[0,1]\times\Gamma\rightarrow\cn$  we can prove a similar result  in the case of a star-shaped tree. Here $\Gamma$ is viewed as a collection of $N$ infinite intervals $(0,\infty)$ coupled at the origin. We consider the following critical exponent
\begin{equation}
\label{gamma}
 \gamma_\Gamma=\frac 12
 \left\{
  \begin{aligned}
  	1, & \quad N\ \text{even},\\
  	\frac{m+1}m, &\quad  N=2m+1.
  \end{aligned}
\right.
\end{equation}

\begin{theorem}
	\label{main-potential}
	Let $\alpha,\beta$ such that $\alpha\beta >4\gamma_\Gamma^4$. Assume the solution $\bu \in C([0,1],L^2(\Gamma))$ of equation
	\begin{equation}
\label{eq.potential}
  \bu_t=i(\Delta_\Gamma+\bv(t,x))\bu \quad \text{in}\ [0,1]\times \Gamma
\end{equation}
satisfies
\begin{equation}
\label{cond.at.zero}
  \|\mathrm{e}^{\alpha x^2 } \bu(0)\|_{L^2(\Gamma)} 
 + \|\mathrm{e}^{\beta x^2 } \bu(1)\|_{L^2(\Gamma)}<\infty,
\end{equation}
where $\bv(t,x)=\bv_1(x)+\bv_2(t,x)$ with $\bv_1$ real-valued, $\|\bv_1\|_{L^\infty(\Gamma)}\le M_1$ and
\[
\sup_{t\in [0,1]}\|\mathrm{e}^{\frac{\alpha\beta|x|^2}{(\sqrt{\alpha} t+(1-t)\sqrt{\beta})^2}}\bv_2(t)\|_{L^\infty(\Gamma)}<+\infty.
\]
Then $\bu$ vanishes identically.
\end{theorem}

The above result is not sharp. In fact when  all the components of $\bv$ are equal, i.e. $V_1\equiv V_2\equiv \dots \equiv V_N$, the result can be improved by using the same strategy as in the proof of Theorem \ref{free-sch} of making the sum of the components and using the real line result. In this case $\gamma_\Gamma$ corresponds to the one in \cite{ekpvsharp}, $\gamma_\Gamma=1/\sqrt{8}$.

\medskip

The paper is organised as follows. In Section \ref{notations} we present the notations and preliminaries about metric graphs and the Schr\"odinger equation on a metric graph. 
In Section \ref{hup-trees} we  
 consider the simple case of a star-shaped tree and give a sketch of how Theorem \ref{free-sch} can be proven in this particular case. Also we show how Theorem \ref{piecewise theorem} implies Theorem \ref{free-sch}.  Theorem \ref{piecewise theorem} is proved in Section \ref{piecewise-section}. Sections  \ref{carleman-tree} and \ref{potential-tree} are devoted to the case of the LSE with a potential on a star-shaped tree.

\section{Notations and Preliminaries }\label{notations}

In this section we present some generalities about metric graphs and introduce the Laplace operator on such structure.
Let $\Gamma=(V,E)$ be a graph where $V$ is a set of vertices and $E$ the set of edges. 
For each $v\in V$ we denote  $E_v=\{e\in E: v\in e\}$.
We assume that $\Gamma$ is a  finite connected graph. The edges could be of finite length and then their ends are vertices of $V$ or they have infinite length and then we assume that each infinite edge is a ray with a single vertex belonging to $V$ (see \cite{MR2459876} for more details on graphs with infinite edges).

We fix an orientation of $\Gamma$ and for each finite oriented edge $e$, we have an initial vertex $I(e)$ and a terminal one $T(e)$. In the case of infinite edges we have only initial vertices.
We identify every edge $e$ of $\Gamma$ with an interval $I_e$, where $I_e=[0,l_e]$ if the edge is finite and $I_e=[0,\infty)$ if the edge is infinite. This identification introduces a coordinate $x_e$ along the edge $e$. In this way $\Gamma$ becomes a metric space, called metric graph  \cite{MR2459876}.
%
%

We identify any function $\bu$ on $\Gamma$ with a collection $\{\ue\}_{e\in E}$ of functions $\ue$ defined on the edges  $e$ of $\Gamma$. Each $\ue$ can be considered as a function on the interval $I_e$. In fact, we use the same notation $\ue$ for both the function on the edge $e$ and the function on the interval $I_e$ identified with $e$.
For a function $\bu:\Gamma\rightarrow \cn$,  $\bu=\{u^e\}_{e\in E}$,  we denote by $f(\bu):\Gamma\rightarrow \cn$ the family 
$(f(u^e))_{e\in E}$, where  $f(u^e):I_e\rightarrow\cn$.

The space $\LL^p(\Gamma)$, $1\leq p<\infty$ consists of all functions   $\bu=\{u_e\}_{e\in E}$ on $\Gamma$ that belong to $\LL^p(I_e)$
for each edge $e\in E$ and 
$$\|\bu\|_{\LL^p(\Gamma)}^p=\sum _{e\in E}\|u^e\|_{\LL^p(I_e)}^p<\infty.$$
Similarly, the space $\LL^\infty(\Gamma)$ consists of all functions that belong to $\LL^\infty(I_e)$ for each edge $e\in E$ and
$$\|\bu \|_{\LL^\infty(\Gamma)}=\max _{e\in E}\|u^e\|_{\LL^\infty(I_e)}<\infty.$$

The Sobolev space $H^m(\Gamma)$, with $m\geq 1$ an integer, consists of all functions with components that belong to
$H^m(I_e)$ for each $e\in E$ and 
$$\|\bu \|_{H^m(\Gamma)}^2=\sum _{e\in E}\|u^e\|_{H^m(e)}^2<\infty.$$
These are  Hilbert spaces with the inner products
$$(\bu,\bsv)_{\LL^2(\Gamma)}=\sum _{e\in E}(\ue,\ve)_{\LL^2(I_e)}=\sum _{e\in E}\int _{I_e}\ue(x)\overline{\ve}(x)\d x$$
and
$$(\bu,\bsv)_{H^m(\Gamma)}=\sum _{e\in E}(\ue,\ve)_{H^m(I_e)}= \sum _{e\in E}\sum _{k=0}^m\int _{I_e} \frac{d^k\ue}{dx^k} \overline{\frac{d^k\ve}{dx^k} }\d x.$$

Notice that a function from $H^m(\Gamma)$ has continuous components on the interior of edges, but there is no information about the continuity at the coupling at the vertices. A function $\bu=\{\ue\}_{e\in E}$ is continuous if and only if $\ue$ is continuous on $\mathring{I}_e$ for every $e\in E$, and moreover, it is continuous at the vertices of $\Gamma$:
$$\ue(v)=u^{e'}(v), \quad \forall \ e,e'\in E_v.$$ 


We introduce the Laplace operator $\Delta_\Gamma$ on the graph $\Gamma$, with Kirchhoff coupling condition. This is a standard procedure and we refer to \cite{MR1476363} for a complete description. The domain of $\Delta_\Gamma$  (see \cite{MR1476363})
 is the space of all continuous functions on $\Gamma$, $\bu=\{u^e\}_{e\in E}$,  such that for every edge $e\in E$,
$u^e\in H^2(I_e)$,  and satisfying the following Kirchhoff-type condition:
$$\sum _{e\in E: T(e)=v} u^e_x(l_e-)=\sum _{e\in E: I(e)=v}u_x^e(0+)  \quad \text{for all} \ v\in V.$$
It acts as the second derivative along the edges
$$(\Delta_\Gamma \bu)^e=(u^e_{xx})\quad \text{for all}\ e\in E, \bu\in D(\Delta_\Gamma).$$

It is easy to verify that $(\Delta_\Gamma, D(\Delta_\Gamma ) )$ is a linear, unbounded, self-adjoint, dissipative operator on $\LL^2(\Gamma)$, i.e. 
$(\Delta_\Gamma\bu,\bu)_{\LL^2(\Gamma)}\leq 0$ for all $\bu\in D(\Delta_\Gamma)$. Since $C_c^\infty(\Gamma)$, the space of functions which are $C^\infty$ on each edge and vanish outside some bounded set of $\Gamma$, is included in $D(\Delta_\Gamma)$ we obtain that $D(\Delta_\Gamma)$ is dense in any $L^p(\Gamma)$, $1\leq p<\infty$. All self-adjoint extensions of the Laplacian on such quantum graphs have been described in \cite{MR2277618} in terms of coupling conditions. Using  the properties of the  operator $\Delta_\Gamma$  we obtain as a consequence of the Hille-Yosida theorem the following well-posedness result.
\begin{theorem}\label{existence}
For any $\bu_0\in D(\Delta_\Gamma)$ there exists a unique solution $\bu(t)$ of system \eqref{eq.tree} that satisfies
$\bu\in C(\rr,D(\Delta_\Gamma))\cap C^1(\rr,\LL^2(\Gamma)).$
Moreover, for any $\bu_0\in \LL^2(\Gamma)$, there exists a unique solution $\bu\in C(\rr,\LL^2(\Gamma))$ that satisfies
$$\|\bu(t)\|_{\LL^2(\Gamma)}= \|\bu_0\|_{\LL^2(\Gamma)}\quad \text{for all}\ t\in \R.$$
\end{theorem}

\section{Schr\"odinger equation with Kirchhoff coupling conditions}\label{hup-trees}
\subsection{The star-shaped tree}
Let us give first a proof of  Theorem \ref{free-sch} in the particular case of a star-shaped tree with $N$ edges, in anticipation of the strategy that one can develop in the case of general regular trees.
For any $\bu_0=(u_{0k})_{k=0}^N\in D(\Delta_\Gamma)$ system \eqref{eq.tree} can be written in an explicit way as follows: $u_k\in C(\R,H^2(0,\infty))\cap C^1(\R,L^2(0,\infty))$,  $k\in\{1,\dots,N\},$
\begin{equation}\label{system1exp}
	\begin{cases}
		i\partial_t u_k+\partial_{xx}u_k=0,& t\neq 0,x>0, k\in \{1,\dots,N\},\\
		u_k(t,0)=u_j(t,0), & k,j\in\{1,\dots,N\},\\
		\sum_{k=1}^n \partial_x u_k(t,0)=0, & t\neq 0.\\
	\end{cases}
\end{equation}
	We can consider the case $\alpha=\beta$, the other case can be reduced to this one by using the so called Appell transformation (see Section \ref{sec-Appell transform} for a precise definition). 
	 Denote by $S$ the sum of all the components of $\bu$:
	\[
	S(t,x)=\sum _{k=1}^N u_k(t,x).
	\]
	It follows that $S$ satisfies the Schr\"odinger equation on the half-line with Neumann boundary condition at $x=0$, $S_x(t,0)=0$.
	Moreover, $S$ satisfies $	|S(0,x)|+|S(1,x)|  \lesssim \mathrm{e}^{-\alpha x^2}$.
	Denoting by $\widetilde S$ the even  extension of $S$ we obtain that it satisfies the Schr\"odinger equation on the whole line
	\[
	i \widetilde S_t +\widetilde S_{xx}=0,\ x\in \rr, \ t \neq 0. 
	\] 
Using the classical result on the real line we conclude that $\widetilde S\equiv 0$ so $S\equiv 0$. Going back to $u_k$, $k=1,\dots, N$ we obtain that each component satisfies the Schr\"odinger equation on the half line with Dirichlet boundary condition at $x=0$, $u_k(t,0)=0$.
Making an odd extension $\widetilde u_k$, one obtains a solution of the linear Schr\"odinger equation on the whole line that decays as follows
$
|\widetilde u_k(1,x)|+|\widetilde u_k(0,x)|\lesssim \mathrm{e}^{-\alpha x^2}.
$
Then $\widetilde u_k\equiv 0$, so $u_k\equiv 0$.

\begin{remark}
This assumption $\alpha\beta>1/16$ is sharp. In the case of the star shaped tree the solution $(u_k)_{k=\overline{1,N}}$ of system \eqref{eq.tree} can be computed explicitly (see \cite{MR2804557} for $N=3$) 
\[
u_k(t,x)=\int_0^\infty k_t(x-y)u_{0k}(y)\d y +\int_0^\infty k_t(x+y)\Big(\frac 2N \sum _{j=1}^N u_{0j}-u_{0k} \Big)(y)\d y,
\]
where $k_t(x)= \frac{1}{\sqrt{4 \pi \mathrm{i} t}} \mathrm{e}^{{\mathrm{i} \frac{x^2}{4 t}}}$.
When $\alpha\beta=1/16$, we consider as initial data
\[
u_{0,k}(x)=\mathrm{e}^{-\alpha x^2-\frac{ix^2}4}, \ k=1,\dots, N.
\]	
Using the fact that all $u_{0,k}$ are equal and the invariance of $\mathrm{e}^{-x^2}$ w.r.t. the Fourier transform, we obtain that for all $1\leq k\leq N$
\[
u_k(1,x)=(k_1\ast u_{0,k})(x)=\frac{1}{\sqrt{2\alpha}}\mathrm{e}^{-\frac{i|x|^2}{4}}\mathrm{e}^{-\frac{|x|^2}{16\alpha}}.
\]

\end{remark}

\subsection{Piecewise constant coeficients and LSE on regular trees}
We will show how one can apply Theorem \ref{piecewise theorem}
in order to obtain the same principle in the case of regular trees with Kirchhoff coupling condition. In order to give a clear and detailed proof, we borrow the notations from \cite{ignatsiam} and recall some of the needed key results. Also, for simplicity, we restrict ourselves to a particular regular tree and we explain the changes that appear in the case of a general regular tree after the proof.

\begin{proof}[Proof of Theorem \ref{free-sch}]

Following \cite{ignatsiam}, we consider the regular tree as in Figure 1, with each internal vertex having other two children nodes, the edges of the same generation have the same length, with the last generation being edges of infinite length.
\begin{center} 
\includegraphics[scale=0.1]{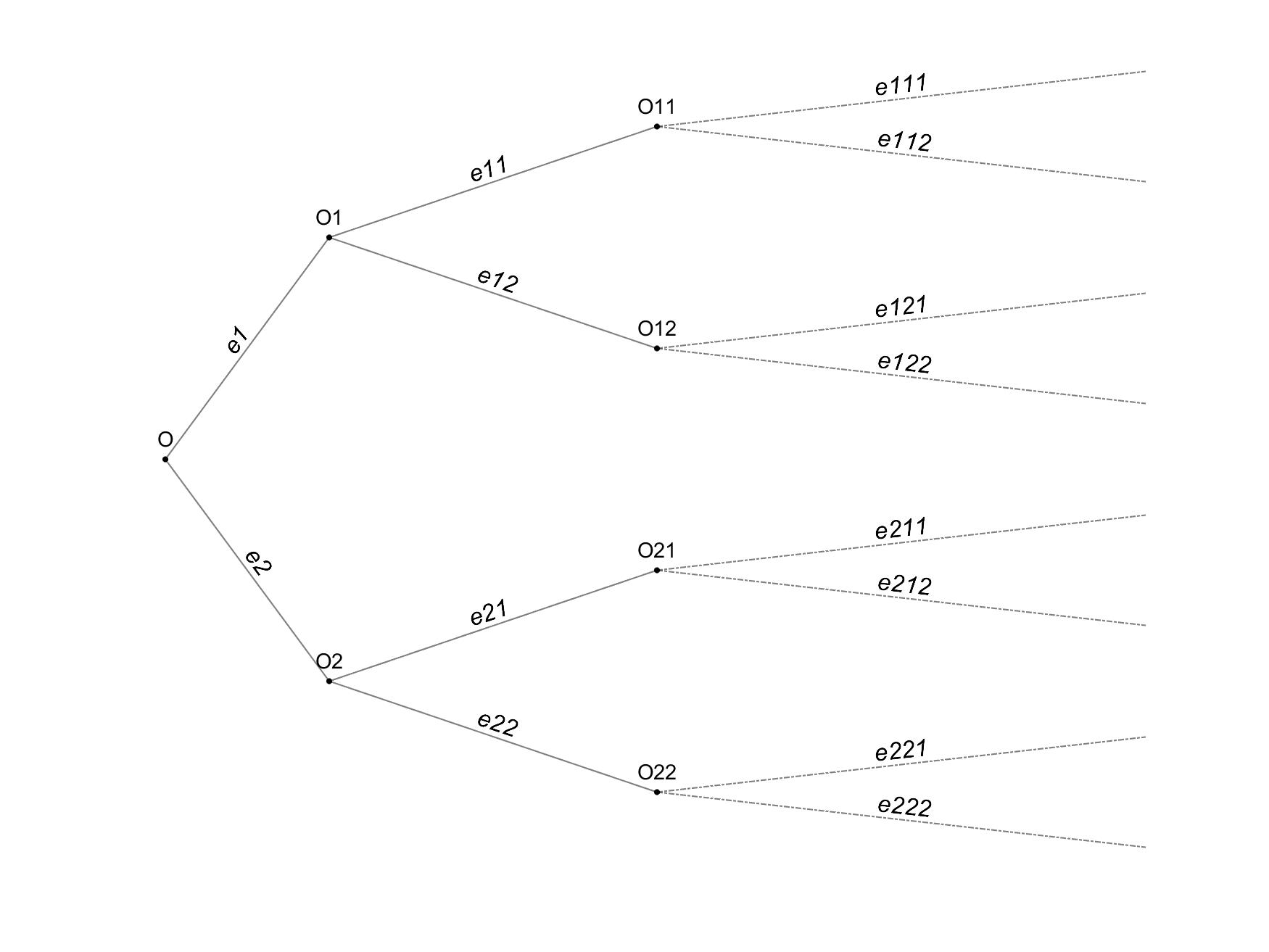}\\
{\textbf{Figure 1.}} \textit{Regular tree with $n+1=3$ generations of edges, $2$ descendants from each vertex.}
\end{center}
Let us assume we have $n$ generations of vertices and, correspondingly, $n+1$ generations of edges, and present their indexing. Consider indices of type $\bar{\alpha}=(\alpha_1,\alpha_2, \dots, \alpha_k) \in \{1,2 \}^k$ and $|\bar{\alpha}|=k$ the number of its components. Denote by $O$ the root of the tree, by $O_{\bar{\alpha}}$ and $e_{\bar{\alpha}}$ the remaining vertices and edges, respectively. From each vertex $O_{\bar{\alpha}}$ with $|\bar{\alpha}| \leq n$ there are two edges that branch out: $e_{\overline{\alpha \beta}}$, with $ \overline{\alpha \beta} = (\alpha_1,\alpha_2, \dots, \alpha_k,\beta), \beta \in \{1,2 \}$. In the case when $|\bar{\alpha}| \leq n-1$, the endpoint of $e_{\overline{\alpha \beta}}$ is $O_{\overline{\alpha \beta}}$, otherwise, i.e. if $|\bar{\alpha}| = n$, the two edges that branch out are infinite strips.

Having these new notations in mind, a function $\bu:\Gamma \to \mathbb{C}$ is a collection of functions $\{ u_{\bar{\alpha}} \}_{\bar{\alpha}}$ defined on each edge, $u_{\bar{\alpha}} : e_{\bar{\alpha}} \to \mathbb{C}$, each edge being identified with the real sub-interval $[0,\l_{|\bar{\alpha}|})$, with $l_{\bar{\alpha}}$ the length of $e_{\bar{\alpha}}$ if $|\bar{\alpha}| \leq n-1$, and $[0, \infty)$ if $|\bar{\alpha}|=n$. 
%
 Denoting by
\begin{equation*}
I_k=\left\{
\begin{array}{lll}
(a_{k-1},a_k),& \text{if } 1 \leq k \leq n, \\[10pt]
(a_n,\infty),& \text{if } k=n+1, 
\end{array}
\right.
\end{equation*}
 with $a_0=0, a_{k+1}=a_k + l_{k+1}, k=\overline{0,n-1}$, $a_{n+1}=\infty$, system \eqref{system1exp} is equivalent (after a space translation) with
\begin{equation}\label{eq3.7siam}
	\begin{cases}
		\mathrm{i} u^{\bar{\alpha}}_t(t,x) +u^{\bar{\alpha}}_{xx}(t,x)=0, & t \neq 0, x \in I_{|\bar{\alpha}|}, 1 \leq |\bar{\alpha}| \leq n+1, \\[10pt]
	u^{\bar{\alpha}}(t,a_{|\bar{\alpha}|})=u^{\overline{\alpha \beta}}(t,a_{|\bar{\alpha}|}), & \beta \in \{ 1,2 \}, 1 \leq |\bar{\alpha}| \leq n,\\[10pt]
	u^1(t,0)=u^2(t,0), \\[10pt]
	u_x^{\bar{\alpha}}(t,a_{|\bar{\alpha}|})=\sum_{\beta=1}^{2}u_x^{\overline{\alpha \beta}}(t,a_{|\bar{\alpha}|}), & 1 \leq |\bar{\alpha}| \leq n,\\[10pt]
u^1_x(t,0)+u^2_x(t,0)=0,\\[10pt]
	u^{\bar{\alpha}}(0,x)=u^{\bar{\alpha}}_0(x).
	\end{cases}
\end{equation}
 For every $\bar{\alpha}$ with $1 \leq |\bar{\alpha}| \leq n+1$, consider the averaged sum functions 
\begin{equation*}
Z^{\bar{\alpha}} : J_{\bar{\alpha}}:= \bigcup_{j=0}^{n+1-|\bar{\alpha}|} I_{|\bar{\alpha}| + j} \to \mathbb{C}
\end{equation*}
 as
\begin{equation*}
Z^{\bar{\alpha}}(t,x) = \dfrac{\sum_{|\bar{\gamma}|=j} u^{\overline{\alpha \gamma}}(t,x)}{2^{j}}, \quad x \in I_{|\bar{\alpha}|+j}, j=\overline{0,n+1-|\bar{\alpha}|}.
\end{equation*}
Note that
\begin{equation}\label{zcoincideu}
Z^{\bar{\alpha}}(\cdot,x)=u^{\bar{\alpha}}(\cdot,x), \quad x \in I_{|\bar{\alpha}|}.
\end{equation}
 Consider now 
\begin{equation*}
Z(t,x)=\dfrac{Z^1(t,x) + Z^2(t,x)}{2}, \quad t \in \mathbb{R}, x \in (0,\infty),
\end{equation*}
 which satisfies ${Z}_x(t,0)=0$, $t \neq 0$, ${Z}(t,a_k-)={Z}(t,a_k+)$ and ${Z}_x(t,a_k-)=2 {Z}_x(t,a_k+)$ for all $ 1 \leq k \leq n$.
Let us introduce the sequence $(\tilde a_k)_{k=0}^{2n+2}$ defined by
\[
\tilde{a}_k=\left\{
\begin{array}{lll}
-a_{n+1-k},& \text{if } 0 \leq k \leq n, \\ 
a_{k-(n+1)},& \text{if } n+1 \leq k \leq 2n+2.
\end{array}
\right.
\]
It follows that  $v(t,x)$,
 the even extension of the function $Z$,  satisfies
\begin{equation}\label{vZevantranslated}
	\begin{cases}
		\mathrm{i} v_t(t,x) +v_{xx}(t,x)=0, & t \neq 0, x \in (\widetilde{a}_k,\widetilde{a}_{k+1}) 1 \leq k \leq 2n, \\[10pt]
		v(t,\widetilde{a}_k-)=v(t,\widetilde{a}_k+), & 1 \leq k \leq 2n+1,\\[10pt]
		v_x(t,\widetilde{a}_k-)=\frac 12 v_x(t,\widetilde{a}_k+), & 1 \leq k \leq n,\\[10pt]
		v_x(t,\widetilde{a}_{n+1}-)=0=v_x(t,\widetilde{a}_{n+1}+), \\[10pt]
		v_x(t,\widetilde{a}_k-)= {2} v_x(t,\widetilde{a}_k+), & n+2 \leq k \leq 2n+1,\\[10pt]
		v(0,x)=v_0(x), &x \in x \in (\widetilde{a}_k,\widetilde{a}_{k+1}) 1 \leq k \leq 2n.
	\end{cases}
\end{equation}

We  consider
\begin{equation*}
w(t,T_k(x))=v(t,x), \quad t\in \mathbb{R}, x \in (\widetilde{a}_k,\widetilde{a}_{k+1}), 0 \leq k \leq 2n+1,
\end{equation*}
where each $T_k:(\widetilde{a}_k,\widetilde{a}_{k+1})\to (b_k,b_{k+1})$, $0 \leq k \leq 2n+1$ is a one-to-one linear map that satisfies $(T_k)_x=\mu_k$. The idea behind this linear transformation is that as long as $v_x(\tilde a_k-)=\eta_k v_x(\tilde a_k+)$ we can construct 
  a piecewise constant coefficient $\sigma$ such that $\sigma(x)=\mu_k^2$ on $(b_k,b_{k+1})$ with $\mu_{k-1}=\mu_k/\eta_k$ and $w$ to satisfy
\begin{equation}
\label{eq.w}
\left\{
\begin{array}{lll}
\mathrm{i} w _t(t,x)+ \partial_x (\sigma \partial_x w))(t,x)=0,& x\in \mathbb{R}, t\neq 0 ,\\
w(0,x)=w_0(x),&   x\in \mathbb{R}.
\end{array}
\right.
\end{equation} 
The particular structure of \eqref{vZevantranslated} where the first half of $\eta$'s are equal with $1/2$ and the second half are equal $2$ allow us 
to consider  $T_k$'s such that $(T_k)_x=2^{|n+1/2-k|-(n+1/2)}$ and 
\[
\sigma(x)= 2^{|2n+1-2k|-(2n+1)}, x\in (b_k,b_{k+1}), \ 0\leq k\leq 2n+1. 
\]

%
%
%
%
%
Recall that $\bu(0,x)={O}(\mathrm{e}^{-\alpha x^2})$ and $\bu(1,x)={O}(\mathrm{e}^{-\beta x^2})$ which implies that
\begin{equation*}
u^{\bar{\alpha}}(0,x)={O}(\mathrm{e}^{-\alpha x^2}) \ \text{and}\
u^{\bar{\alpha}}(1,x)={O}(\mathrm{e}^{-\beta x^2}), \, \forall \ |\bar{\alpha}|\leq n+1,
\end{equation*}
which in view of the previous arguments, one gets that
\begin{equation*}
w(0,x)={O}(\mathrm{e}^{-\alpha  {x^2} })
\ \text{and}\ w(1,x)={O}(\mathrm{e}^{-\beta  {x^2} }).
\end{equation*} 
Thus, since from the definition of $\sigma$ we get that $\sigma_{\pm}=1$, by Theorem \ref{piecewise}, $w(t,x)=0$, for $x \in  \mathbb{R}$ and $t \in [0,1]$. This implies that 
\begin{equation} \label{zequal0}
Z(t,x)=\dfrac{Z^1(t,x)+Z^2(t,x)}{2}=0, \quad x \in [0, \infty), \quad t \in [0,1].
\end{equation}

 We would like to conclude that $Z^1$ and $Z^2$ vanish for $x \in [0,\infty)$ and $t \in [0,1]$, and thus, by \eqref{zcoincideu}, $u^1$ and $u^2$ vanish for $x \in I_1$ and $t \in [0,1]$. Consider, as in \cite{ignatsiam}, the difference functions
\begin{equation}\label{ztilde}
\widetilde{Z}^1=Z-Z^1 \text{ and } \widetilde{Z}^2=Z-Z^2.
\end{equation}
 Since $\widetilde{Z}^1(t,0)=0$,
making an odd extension $\widetilde{Z}^{1,odd}$ to the whole real line, it satisfies an equation similar to \eqref{vZevantranslated} except the fact that $v_x(t,\widetilde{a}_{n+1}-) =v_x(t,\widetilde{a}_{n+1}+)$ not necessarily vanishes and  
the initial data is  $\widetilde{Z}^{1,odd}(0,x)$. 
Since $\widetilde{Z}^{1,odd}(0,x)$ and $\widetilde{Z}^{1,odd}(1,x)$ are again of order ${O}(\mathrm{e}^{-\alpha x^2})$ and ${O}(\mathrm{e}^{-\beta x^2})$, respectively, repeating the previous steps, one arrives finally to the conclusion that $\widetilde{Z}^{1}$ vanishes for all $x \in [0,\infty)$ and $t\in [0,1]$. Together with \eqref{zequal0} and \eqref{ztilde}, we get that
$
Z^1(t,x)=0$, for $x \in [0,\infty)$ and  $t \in [0,1].$
 Similarly, one gets that
$
Z^2(t,x)=0,$ for  $x \in [0,\infty)$ and  $t \in [0,1].
$
 Thus, 
\begin{equation*}
u^1(t,x)=0 \text{ and } u^2(t,x)=0, \quad x \in I_1, \quad t \in [0,1].
\end{equation*}

 The vanishing property for the other components $u^{\bar{\alpha}}$, $1<|\bar{\alpha}| \leq n+1$, follows by induction. More precisely, 
assume that $Z^{\bar{\alpha}}$ vanishes for $x \in J_{\bar{\alpha}}$ and $t \in [0,1]$, for some $|\bar{\alpha}|=k$, and consider the difference functions
\begin{equation*}
\widetilde{Z}^{\overline{\alpha \beta}}(t,x)=Z^{\overline{\alpha \beta}}(t,x) - Z^{\bar{\alpha}}(t,x), \quad x \in J_{k+1}.
\end{equation*}
It follows (see for example \cite{ignatsiam}) that for $k \leq n-1$, $\widetilde{Z}^{\overline{\alpha \beta}}$ satisfies
\begin{equation}\label{eq3.19siam}
	\begin{cases}
		\mathrm{i} \widetilde{Z}_t^{\overline{\alpha \beta}}(t,x) +\widetilde{Z}_{xx}^{\overline{\alpha \beta}}(t,x)=0, & t \neq 0, x \in \bigcup_{m=k+1}^{n+1} I_m, \\[10pt]
 \widetilde{Z}^{\overline{\alpha \beta}}(t,a_k)=0, & t \neq 0,\\[10pt]
		\widetilde{Z}^{\overline{\alpha \beta}}(t,a_m-)=\widetilde{Z}^{\overline{\alpha \beta}}(t,a_m+), & k+1 \leq m \leq n,\\[10pt]
		\widetilde{Z}_x^{\overline{\alpha \beta}}(t,a_m-)=2\widetilde{Z}_x^{\overline{\alpha \beta}}(t,a_m+), & k+1 \leq m \leq n,\\[10pt]
		\widetilde{Z}^{\overline{\alpha \beta}}(0,x)=\widetilde{Z}^{\overline{\alpha \beta}}_0(x), &x \in \bigcup_{m=k+1}^{n+1} I_m,
	\end{cases}
\end{equation}
 and if $k=n$
\begin{equation}\label{eq3.20siam}
\begin{cases}
		\mathrm{i} \widetilde{Z}_t^{\overline{\alpha \beta}}(t,x) +\widetilde{Z}_{xx}^{\overline{\alpha \beta}}(t,x)=0, & t \neq 0, x \in I_{n+1}, \\[10pt]
 \widetilde{Z}^{\overline{\alpha \beta}}(t,a_k)=0, & t \neq 0,\\[10pt]
		\widetilde{Z}^{\overline{\alpha \beta}}(0,x)=\widetilde{Z}^{\overline{\alpha \beta}}_0(x), &x \in I_{n+1}.
	\end{cases}
\end{equation}
If $k \leq n-1$, after a translation to move the point $x=a_k$ to the origin $x=0$, 
proceeding similarly as in the case of $\widetilde{Z}^1$, one finally gets that
\begin{equation*}
Z^{\overline{\alpha \beta}}(t,x)=0, \quad x \in J_{k+1}, \quad t \in [0,1],
\end{equation*}
 and thus,
\begin{equation*}
u^{\overline{\alpha \beta}}(t,x)=0, \quad x \in I_{k+1}, \quad t \in [0,1].
\end{equation*} 
 If $k=n$, making an odd extension of $\widetilde{Z}^{\overline{\alpha \beta}}$ to the whole real line, since $\widetilde{Z}^{\overline{\alpha \beta}}(0,x) = {O}(\mathrm{e}^{-\alpha x^2})$ and $\widetilde{Z}^{\overline{\alpha \beta}}(1,x)={O}(\mathrm{e}^{-\beta x^2})$, by the classical Hardy uncertainty principle for the LSE on $\mathbb{R}$ follows the desired result.
\end{proof}

 \textbf{Extension to general regular trees.} In the proof of Theorem \ref{free-sch}, the regular tree was assumed such that all vertices have two descendants. Let us review the modifications that appear in the case of a regular tree with all the vertices from the $0 \leq k \leq n$ generation having $d
_{k+1}$ descendants (edges). In this case, the indexing is of type $\bar{\alpha}=(\alpha
_1,\alpha_2, \dots, \alpha_k) \in \{1, 2, \dots, \d_1 \} \times  \{1, 2, \dots, \d_2 \} \times \cdots  \times \{1, 2, \dots, \d_k \} $. From each vertex $O_{\bar{\alpha}}$ with $|\bar{\alpha}|=k \leq n$ there are $d_{k+1}$ edges that branch out, $e_{\overline{\alpha \beta_{k+1}}}$, with $\overline{\alpha \beta_{k+1}} = (\alpha_1, \alpha_2, \dots, \alpha_k, \beta_{k+1})$, with $\beta_{k+1} \in \{1, 2, \dots, \d_{k+1} \}$, having endpoints $O_{\alpha \beta_{k+1}}$, and if $|\bar{\alpha}| = n$, the $d_{n+1}$ that branch out from the vertices of the last generation, are infinite strips. In this view, the function $\{ u_{\bar{\alpha}} \}_{\bar{\alpha}}$ modifies accordingly. Furthermore, in equation \eqref{eq3.7siam}, for each $|\bar{\alpha}| =k $, $\beta$ replaces with $\beta_{k+1}$ and the sums are indexed by $\beta_{k+1}=\overline{1,d_{k+1}}$.

The averaged sum functions become, for each $|\bar{\alpha}|=k$,
\begin{equation*}
Z^{\bar{\alpha}}(t,x) = \dfrac{\sum_{|\bar{\gamma}|=j} u^{\overline{\alpha \gamma}}(t,x)}{d_{|\bar{\alpha}|+1},\cdot \dots \cdot d_{|\bar{\alpha}|+k}}, \quad x \in I_{|\bar{\alpha}|+j}, j=\overline{0,n+1-|\bar{\alpha}|}.
\end{equation*}
The new function $Z$ is 
\begin{equation*}
Z(t,x)=\dfrac{Z^1(t,x) + Z^2(t,x)+ \dots + Z^{d_1}(t,x)}{d_1}, \quad t \in \mathbb{R}, x \in (0,\infty),
\end{equation*}
and its even extension satisfies system \eqref{vZevantranslated} with initial data modified accordingly and with $ d_{k+1}$ instead of $ 2$. This latter modification is carried out then throughout the entire proof. In particular, one finally arrives to the step function 
\begin{equation*}
\sigma(x)=
\begin{cases}
\bigg(\dfrac{ d_{2} \cdots d_{n+1-k} }{ d_2   \cdots d_{n+1}} \bigg)^2, & x\in (b_k,b_{k+1}),  0 \leq k \leq n-1, \\
(  d_2 \cdots  d_{n+1})^{-2}, & x\in (b_k,b_{k+1}),  n \leq k \leq n+1, \\
\bigg(\dfrac{ d_{2} \ldots d_{k-n} }{ d_2 \cdots d_{n+1}} \bigg)^2, & x\in (b_k,b_{k+1}),  n+2 \leq k \leq 2n+1,
\end{cases}
\end{equation*}
Taking then
$
\widetilde{Z}^j=Z-Z^j, \quad j=\overline{1,d_1},
$
 the proof follows similarly, keeping these changes accordingly. Again, in this case $\sigma_{\pm}=1$.

\section{Proof of Theorem \ref{piecewise}}\label{piecewise-section}

Let $N \geq 2$, and consider a partition of the real line $l_0 = -\infty <l_1 <l_2 < \cdots < l
_{N-1} < l_N =\infty$, and on each interval $I_i=(l_{i-1},l_i)$ assume that $\sigma(x)=a_i^{-2}$ is constant, with $a_i>0$, for all $i=\overline{1,N}$. Since one can always reduce to the case $l_1=0$ and assume that $l_j=(j-1)l$ by refining the partition of the $l_j$'s, from now on we will consider this special partition of $\mathbb{R}$. This special choice will be used in this section.

Let $u_0$ be as in Theorem \ref{piecewise theorem}. Then it belongs to $L^2(\mathbb{R})$. Since the operator $L:\mathcal{D}(L) \subseteq L^2(\mathbb{R}) \to L^2(\mathbb{R})$, acting as $Lu=\partial_x (\sigma \partial_x u)$ is self-adjoint, it generates a unitary group, i.e. there exists a unique mild solution in $L^2(\mathbb{R})$ for \eqref{eq.tree}.
 
  In order to prove Theorem \ref{piecewise}, we would like first to obtain an explicit representation of the solution $u$ of equation \eqref{eq.sigma} with $\sigma$ as explained at the beginning of this section. For reasons which will be clearer in what follows, it is sufficient to compute the solution only for $x \leq 0$. Following \cite{gaveauokada87}, $u$ can be written for all $t \neq 0$ and $x \leq 0$ as 
\begin{equation} \label{solution piecewise}
u(t,x)=\int_{-\infty}^0 p_t^{1,1}(x,y)u_0(y) \, \d y + \sum_{j=2}^{N-1} \int_{(j-2)l}^{(j-1)l} p_t^{1,j}(x,y) u_0(y) \, \d y + \int_{(N-2)l}^{\infty} p_t^{1,N}(x,y)u_0(y) \, \d y, 
\end{equation} 
where 
\begin{equation} \label{kernel piecewise}
p_t^{1,j}(x,y)=\int_{\mathbb{R}} \mathrm{e}^{- \mathrm{i} \xi^2 t} \big[ C^{-}_{1j}(\xi) \mathrm{e}^{\mathrm{i} \xi (a_1 x - a_j y)} + C^{+}_{1j}(\xi) \mathrm{e}^{\mathrm{i} \xi (a_1 x + a_j y)} \big] \, d \xi, \quad y \in I_j, \ j=\overline{1,N},
\end{equation}
 with 
\begin{equation}\label{c11c1n}
C^{-}_{11}(\xi)=\frac{a_1}{2 \pi}, \quad C^{+}_{1N}(\xi)=0
\end{equation}
  and 
\begin{equation} \label{cij via tjs}
\begin{bmatrix} 
    C^{-}_{1j}(\xi) \\
    C^{+}_{1j}(\xi) \\
    \end{bmatrix}
 = \bar{T}_{j-1}(\xi) \cdot \dots \cdotp \bar{T}_{1}(\xi) 
\begin{bmatrix} 
    C^{-}_{11}(\xi) \\
    C^{+}_{11}(\xi) \\
    \end{bmatrix}.
\end{equation}
The complexly conjugated matrices $\bar{T}_j(\xi)$, $1\leq j
\leq N-1$, are of the type
\begin{equation}\label{transfer_matrices}
\bar{T}_j(\xi)= \frac{\varepsilon_j}{2 a_j}
\begin{bmatrix}
\mathrm{e}^{-\mathrm{i} \xi \delta_j (j-1) l} & \gamma_j \mathrm{e}^{\mathrm{i} \xi \varepsilon_j (j-1) l}\\
\gamma_j \mathrm{e}^{-\mathrm{i} \xi \varepsilon_j (j-1) l} & \mathrm{e}^{\mathrm{i} \xi \delta_j (j-1) l} \\
\end{bmatrix}
=: \frac{\varepsilon_j}{2 a_j}
\begin{bmatrix}
\lambda _j(\xi) & \bar{\mu}_j(\xi) \\
\mu_j(\xi) & \bar{\lambda}_j (\xi)\\
\end{bmatrix}
\end{equation}
with 
\begin{align}\label{deltaepsilongamma}
\delta_j:=a_j - a_{j+1}, \quad  \varepsilon_j: = a_j + a_{j+1}, \quad \gamma_j := \frac{\delta_j}{\varepsilon_j} \\
\lambda_j(\xi)= \mathrm{e}^{-\mathrm{i} \xi \delta_j (j-1) l}, \quad \mu_j(\xi)= \gamma_j \mathrm{e}^{-\mathrm{i} \xi \varepsilon_j (j-1) l}. \label{lambdamu}
\end{align}
In particular, taking $j=N$ in \eqref{cij via tjs} and using \eqref{c11c1n}
\begin{equation}\label{c1nminus}
\begin{bmatrix} 
    C^{-}_{1N}(\xi) \\
    0 \\
    \end{bmatrix}
 = \bar{T}_{N-1}(\xi) \cdot \dots \cdotp \bar{T}_{1}(\xi) 
\begin{bmatrix} 
    a_1/(2 \pi) \\
    C^{+}_{11}(\xi) \\
    \end{bmatrix}.
\end{equation}
This allows us to obtain $C^+_{11}$ and $C^-_{1N}$ in terms of the matrices $\bar{T}_k$, $k=\overline{1,N}$. Thus, any other $C^{\pm}_{1k}$ can be written as
\begin{equation}\label{cikminusplusinverse}
\begin{bmatrix} 
    C^{-}_{1k}(\xi) \\
    C^{+}_{1k}(\xi) \\
    \end{bmatrix}
 = \big( \bar{T}_{N-1}(\xi) \cdot \dots \cdotp \bar{T}_{k}(\xi) \big)^{-1}
\begin{bmatrix} 
    C^{-}_{1N}(\xi) \\
    0 \\
    \end{bmatrix}, \quad k=\overline{1,N-1}.
\end{equation}

Our aim is to deduce first an easier way to handle the expression for the coefficients $C^{\mp}_{1,k}(\xi)$, for $k=1,\dots,N$. While one can compute them inspired by \cite{gaveauokada87}, for clarity of the exposition we prefer to make it explicit. Their representation will be given in terms of  two  sequences of functions $E_{j,k}(\xi)$, $F_{j,k}(\xi)$, $1\leq k\leq j\leq N-1$ defined as follows:
 \begin{equation}\label{ejftildej}
 \left\{
  \begin{aligned}
& E_{j,k}(\xi) =  E_{j-1,k}(\xi) + \gamma_j \mathrm{e}^{2 \mathrm{i} \xi l (a_{k+1} + \dots + a_{j})} \widetilde{F}_{j-1,k}(\xi) \\
& \widetilde{F}_{j,k}(\xi) =  \widetilde{F}_{j-1,k}(\xi) + \gamma_j \mathrm{e}^{-2 \mathrm{i} \xi l (a_{k+1} + \dots + a_{j})} E_{j-1,k}(\xi)
\end{aligned}, j>k,
\right.
\end{equation}
and $E_{k,k}(\xi)=1$, $\widetilde{F}_{k,k}(\xi)=\gamma_k.$
Recursively, we get  for all $ i_{k+1},\dots,i_j \in  \{0, 1 \}$ the existence of some constants $c_{i_{k+1},\dots,i_j}$, $\tilde{c}_{i_{k+1},\dots,i_j}$,  such that for all $1\leq k<j\leq N-1$
\begin{equation}\label{ejftildejsums}
 \left\{
  \begin{aligned}
& E_{j,k}(\xi) =  \sum_{i_{\square} \in \{ 0,1 \}} c_{i_{k+1},\dots,i_j} \mathrm{e}^{ 2 \mathrm{i} \xi l (i_{k+1} a_{k+1} + \dots i_j a_j)}, \\
& \widetilde{F}_{j,k}(\xi) =  \sum_{i_{\square} \in \{ 0,1 \}} \tilde{c}_{i_{k+1},\dots,i_j} \mathrm{e}^{- 2 \mathrm{i} \xi l (i_{k+1} a_{k+1} + \dots i_j a_j)},  
 \end{aligned}
\right.
\end{equation}
where $\sum_{i_{\square} \in \{ 0,1 \}} = \sum_{i_{k+1} \in \{ 0,1 \}} \ldots \sum_{i_{j} \in \{ 0,1 \}}$, depending on the indexes $i_n$ appearing in the complex exponentials.

In this view, let us first prove the following lemma.
\begin{lemma}\label{tjtk}
Let $\lambda_1, \mu_1, \ldots, \lambda_{N-1}, \mu_{N-1}$ as in \eqref{transfer_matrices}. For any $1 \leq k < j \leq N-1$,
\begin{align*}
  [ \bar{T}_{j}(\xi) \cdot \dots \cdotp \bar{T}_{k}(\xi) ]_{21} &= \frac{\varepsilon_k \cdot \ldots \cdotp \varepsilon_j}{2^{j-k+1} a_k \cdot \dots \cdotp a_j} \bar{\lambda}_{k}(\xi) \cdot \dots \cdotp \bar{\lambda}_j(\xi) \mathrm{e}^{- 2 \mathrm{i} \xi l (k-1) a_k}  \widetilde F_{j,k},
\\
 [ \bar{T}_{j}(\xi) \cdot \dots \cdotp \bar{T}_{k}(\xi) ]_{22} &= \frac{\varepsilon_k \cdot \ldots \cdotp \varepsilon_j}{2^{j-k+1} a_k \cdot \dots \cdotp a_j} \bar{\lambda}_{k}(\xi) \cdot \dots \cdotp \bar{\lambda}_j(\xi) \overline{E}_{j,k}.
\end{align*}
%
\end{lemma}
\begin{proof}[Proof of Lemma \ref{tjtk}]
Taking into account the form of the matrices $\bar{T}_j(\xi)$ in \eqref{transfer_matrices}, we observe that their product is of type 
\begin{equation}\label{producttjtk}
 \bar{T}_{j}(\xi) \cdot \dots \cdotp \bar{T}_{k}(\xi) = : 
 \begin{bmatrix}
 A_{j,k}(\xi) & \bar{B}_{j,k}(\xi) \\
 B_{j,k}(\xi) & \bar{A}_{j,k}(\xi) \\
 \end{bmatrix}.
\end{equation}
In the following, we obtain an explicit representation of the matrices $A_{j,k}$ and $B_{j,k}$. 
 For any $j>k$ we have
\begin{equation}\label{ajbj}
\begin{bmatrix}
A_{j,k}(\xi) \\
B_{j,k}(\xi) \\
\end{bmatrix}
 = \frac{\varepsilon_j}{2 a_j} 
\begin{bmatrix}
\lambda _j(\xi) & \bar{\mu}_j(\xi) \\
\mu_j(\xi) & \bar{\lambda}_j(\xi) \\
\end{bmatrix} 
\begin{bmatrix}
A_{j-1,k}(\xi) \\
B_{j-1,k}(\xi) \\
\end{bmatrix},
\end{equation}
with
\begin{equation}\label{akbk}
\begin{bmatrix}
A_{k,k}(\xi) \\
B_{k,k}(\xi) \\
\end{bmatrix}
= \frac{\varepsilon_k}{2 a_k} 
\begin{bmatrix}
\lambda_k(\xi) \\
\mu_k(\xi) \\
\end{bmatrix}.
\end{equation}
Denoting
\begin{equation}\label{ajbjwithejfj}
 \left\{
  \begin{aligned}
  A_{j,k}(\xi) = \frac{\varepsilon_k \cdot \ldots \cdotp \varepsilon_j}{2^{j-k+1} a_k \cdot \dots \cdotp a_j}  {\lambda}_{k}(\xi) \cdot \dots \cdotp {\lambda}_j(\xi) E_{j,k}(\xi) \\
  B_{j,k}(\xi) = \frac{\varepsilon_k \cdot \ldots \cdotp \varepsilon_j}{2^{j-k+1} a_k \cdot \dots \cdotp a_j} \bar{\lambda}_{k}(\xi) \cdot \dots \cdotp \bar{\lambda}_j(\xi) F_{j,k}(\xi)
  \end{aligned}
\right.
\end{equation}
and using that $\lambda_j(\xi)=\mathrm{e}^{-\mathrm{i} \xi \delta_j (j-1) l}$ and $\mu_j(\xi)=\gamma_j \mathrm{e}^{-\mathrm{i} \xi \varepsilon_j (j-1) l}$, with $\delta_j, \varepsilon_j$ and $\gamma_j$ as in \eqref{deltaepsilongamma}, 
we obtain that, in order to verify \eqref{ajbj}, $E_{j,k}$ and $F_{j,k}$ satisfy
\begin{equation*}
 \left\{
  \begin{aligned}
& E_{j,k}(\xi) =  E_{j-1,k}(\xi) + \gamma_j \mathrm{e}^{2 \mathrm{i} \xi l (a_{k+1} + \dots + a_{j})} \cdot \mathrm{e}^{2 \mathrm{i} \xi l (k-1) a_k}  F_{j-1,k}(\xi) \\
& F_{j,k}(\xi) =  F_{j-1,k}(\xi) + \gamma_j \mathrm{e}^{-2 \mathrm{i} \xi l (a_{k+1} + \dots + a_{j})} \cdot \mathrm{e}^{-2 \mathrm{i} \xi l (k-1) a_k}  E_{j-1,k}(\xi)
\end{aligned}
\right.
\end{equation*}
with $E_{k,k}(\xi)=1,\ F_{k,k}(\xi)=\gamma_k\mathrm{e}^{-2 \mathrm{i} \xi l (k-1)a_k}$.  Introducing 
$\widetilde{F}_{j,k}(\xi)=\mathrm{e}^{2 \mathrm{i} \xi l (k-1)a_k} F_{j,k}(\xi)$
we obtain \eqref{ejftildej}.

 In view of the notations in \eqref{producttjtk}
\[
 \big[ \bar{T}_{j}(\xi) \cdot \dots \cdotp \bar{T}_{k}(\xi) \big]_{21} = {B}_{j,k}(\xi), \big[ \bar{T}_{j}(\xi) \cdot \dots \cdotp \bar{T}_{k}(\xi) \big]_{22} = \bar{A}_{j,k}(\xi) 
 \]
 the desired result follows by using \eqref{ajbjwithejfj}.
\end{proof}

We  now give the precise expansions of  the coefficients $C^{\mp}_{1k}(\xi)$, $k=\overline{1,N}$, besides $C^{-}_{11}(\xi)$ and $C^{+}_{1N}(\xi)$ in \eqref{c11c1n}. In order to enlighten the expressions, set
\begin{equation} \label{alpha_k coefficient}
\alpha_k:=\dfrac{\varepsilon_1 \cdot \ldots \cdotp \varepsilon_{k-1}}{2^{k-1} a_1 \cdot \ldots \cdotp a_{k-1}}(1-\gamma_1^2)\cdot \dots \cdotp (1-\gamma_{k-1}^2), \quad 2 \leq k\leq N.
\end{equation}

\begin{lemma} \label{preliminary_coeff}
 The coefficients $C^{\mp}_{1k}(\xi)$,  can be written as follows
\[
 C^{-}_{11}(\xi) = \dfrac{a_1}{2 \pi},\quad  C^{+}_{11}(\xi)=-C^{-}_{11}(\xi) \frac{\widetilde F_{N-1,1}(\xi)}{\overline{E}_{N-1,1}(\xi)},
 \]
\[
 C^{-}_{1N}(\xi)= \frac{\alpha_{N}}{{\overline{\lambda}_1(\xi)\cdot\ldots\cdot\overline{\lambda}_{N-1}(\xi)}}    \frac{ C^{-}_{11}(\xi)}{\overline{E}_{N-1,1}(\xi)}, \quad  C^{+}_{1N}(\xi) = 0,
\]
and for $2 \leq k \leq N-1$,
\[
 C^{-}_{1k}(\xi)= \frac{\alpha_k}{{\overline{\lambda}_1(\xi)\cdot\ldots\cdot\overline{\lambda}_{k-1}(\xi)}} \cdot C^{-}_{11}(\xi) \cdot \frac{\overline{E}_{N-1,k}(\xi)}{\overline{E}_{N-1,1}(\xi)},
 \]
 \[  C^{+}_{1k}(\xi)= - \frac{\alpha_k}{{\overline{\lambda}_1(\xi)\cdot\ldots\cdot\overline{\lambda}_{k-1}(\xi)}} \cdot C^{-}_{11}(\xi) \cdot \mathrm{e}^{- 2 \mathrm{i} \xi l (k-1) a_k} \frac{\widetilde F_{N-1,k}(\xi)}{\overline{E}_{N-1,1}(\xi)}. 
\]
\end{lemma}


\begin{proof}
We first emphasize that $\big[ \bar{T}_{N-1}(\xi) \cdot \dots \cdotp \bar{T}_{1}(\xi) \big]_{22}\neq 0$ for all $\xi\in \rr$. This follows from \eqref{c1nminus} since $a_1\neq 0$ (otherwise all coefficients in the matrix would be 0 and the matrix would not be invertible). This implies that $A_{N-1,1}(\xi)$ and $E_{N-1,1}(\xi)$ do  not vanish  on the real line.
From \eqref{c1nminus}, we get immediately the expression for $C^{+}_{11}(\xi)$. We deduce, substituting $C^{+}_{11}(\xi)$, that
\begin{equation}\label{c1Nminus}
C^{-}_{1N}(\xi)=C^{-}_{11}(\xi) \cdot \frac{|A_{N-1,1}(\xi)|^2-|B_{N-1,1}(\xi)|^2}{\bar{A}_{N-1,1}(\xi)}.
\end{equation}
 Notice that due to \eqref{ajbj}, we have for all $1 \leq k < j \leq N-1$
\begin{align}\label{determinant}
|A_{j,k}(\xi)|^2-&|B_{j,k}(\xi)|^2  = \Big(\frac{\varepsilon_j}{2 a_j}\Big)^2 \cdot (1-\gamma^2_j) \big[ |A_{j-1,k}(\xi)|^2-|B_{j-1,k}(\xi)|^2 \big] = \ldots = \nonumber \\
& = \frac{(\varepsilon_{k+1} \cdot \ldots \cdotp \varepsilon_j)^2}{(2^{j-k} a_{k+1} \cdot \dots \cdotp a_j)^2 } (1-\gamma^2_{k+1}) \cdot \dots \cdotp (1-\gamma^2_j) \big[ |A_{k,k}(\xi)|^2-|B_{k,k}(\xi)|^2 \big] \nonumber \\
&  = \frac{(\varepsilon_{k} \cdot \ldots \cdotp \varepsilon_j)^2}{(2^{j-k+1} a_k \cdot \dots \cdotp a_j )^2} (1-\gamma^2_k) \cdot \dots \cdotp (1-\gamma^2_j),
\end{align}
 where the last inequality follows from \eqref{akbk}. Thus, \eqref{c1Nminus} rewrites as
\begin{equation*}
C^{-}_{1N}(\xi)= C^{-}_{11}(\xi) \cdot \frac{(\varepsilon_1 \cdot \ldots \cdotp \varepsilon_{N-1})^2}{(2^{N-1} a_1 \cdot \dots \cdotp a_{N-1})^2}  (1-\gamma^2_1) \cdot \dots \cdotp (1-\gamma^2_{N-1}) \cdot \frac{1}{\big[ \bar{T}_{N-1}(\xi) \cdot \dots \cdotp \bar{T}_{1}(\xi) \big]_{22}}.
\end{equation*}

 Let us now handle the coefficients $C^{\mp}_{1k}(\xi)$, $2 \leq k \leq N-1$. By \eqref{producttjtk}, one can check that
\begin{align*}
(\bar{T}_{j}(\xi) \cdot \dots \cdotp& \bar{T}_{k}(\xi))^{-1}  =\frac{1}{|A_{j,k}(\xi)|^2-|B_{j,k}(\xi)|^2} \cdot
 \begin{bmatrix}
 \bar{A}_{j,k}(\xi) & -\bar{B}_{j,k}(\xi) \\
 -B_{j,k}(\xi) & \bar{A}_{j,k}(\xi) \\
 \end{bmatrix} \\
 & = \frac{(2^{j-k+1} a_k \cdot \dots \cdotp a_j)^2 }{(\varepsilon_k \cdot \ldots \cdotp \varepsilon_j)^2} \cdot \dfrac{1}{(1-\gamma^2_k) \cdot \dots \cdotp (1-\gamma^2_j)}\cdot 
 \begin{bmatrix}
 \bar{A}_{j,k}(\xi) & -\bar{B}_{j,k}(\xi) \\
 -B_{j,k}(\xi) & \bar{A}_{j,k}(\xi) \\
 \end{bmatrix},
\end{align*}
 where the last identity follows by \eqref{determinant}. Taking now $j=N-1$ in the above relation, by \eqref{cikminusplusinverse} we obtain
\begin{equation*}
\begin{bmatrix} 
    C^{-}_{1k}(\xi) \\
    C^{+}_{1k}(\xi) \\
    \end{bmatrix}
 = \frac{(2^{N-k} a_k \cdot \dots \cdotp a_{N-1})^2 }{(\varepsilon_k \cdot \ldots \cdotp \varepsilon_{N-1})^2(1-\gamma^2_k) \cdot \dots \cdotp (1-\gamma^2_{N-1})} \begin{bmatrix}
 \bar{A}_{N-1,k}(\xi) & -\bar{B}_{N-1,k}(\xi) \\
 -B_{N-1,k}(\xi) & \bar{A}_{N-1,k}(\xi) \\
 \end{bmatrix} \cdot
\begin{bmatrix} 
    C^{(-)}_{1N}(\xi) \\
    0 \\
\end{bmatrix}, 
\end{equation*}
 which implies together with \eqref{producttjtk} that  for $ 2 \leq k \leq N-1$
\begin{equation*}
\left\{
\begin{aligned}
& C^{-}_{1k}(\xi)= \frac{(2^{N-k} a_k \cdot \dots \cdotp a_{N-1})^2 }{(\varepsilon_k \cdot \ldots \cdotp \varepsilon_{N-1})^2(1-\gamma^2_k) \cdot \dots \cdotp (1-\gamma^2_{N-1})} {\big[ \bar{T}_{N-1}(\xi) \cdot \dots \cdotp \bar{T}_{k}(\xi) \big]_{22}} \cdot C^{-}_{1N}(\xi),\\
& C^{+}_{1k}(\xi)=  -  \frac{(2^{N-k} a_k \cdot \dots \cdotp a_{N-1})^2 }{(\varepsilon_k \cdot \ldots \cdotp \varepsilon_{N-1})^2(1-\gamma^2_k) \cdot \dots \cdotp (1-\gamma^2_{N-1})} {\big[ \bar{T}_{N-1}(\xi) \cdot \dots \cdotp \bar{T}_{k}(\xi) \big]_{21}} \cdot C^{-}_{1N} (\xi).
\end{aligned}
\right.
\end{equation*}
 Substituting now the previously obtained expression for $C^{-}_{1N}(\xi)$, we get for $ 2 \leq k \leq N-1$
\begin{equation*}
\left\{
\begin{aligned}
& C^{-}_{1k}(\xi)= C^{-}_{11}(\xi) \cdot \frac{ (\varepsilon_1 \cdot \ldots \cdotp \varepsilon_{k-1})^2  (1-\gamma^2_1) \cdot \dots \cdotp (1-\gamma^2_{k-1})}{(2^{k-1} a_1 \cdot \dots \cdotp a_{k-1})^2} \cdot \frac{\big[ \bar{T}_{N-1}(\xi) \cdot \dots \cdotp \bar{T}_{k}(\xi) \big]_{22}}{\big[ \bar{T}_{N-1}(\xi) \cdot \dots \cdotp \bar{T}_{1}(\xi) \big]_{22}},\\
& C^{+}_{1k}(\xi)= - C^{-}_{11}(\xi) \cdot \frac{ (\varepsilon_1 \cdot \ldots \cdotp \varepsilon_{k-1})^2(1-\gamma^2_1) \cdot \dots \cdotp (1-\gamma^2_{k-1})  }{(2^{k-1} a_1 \cdot \dots \cdotp a_{k-1})^2} \cdot \frac{\big[ \bar{T}_{N-1}(\xi) \cdot \dots \cdotp \bar{T}_{k}(\xi) \big]_{21}}{\big[ \bar{T}_{N-1}(\xi) \cdot \dots \cdotp \bar{T}_{1}(\xi) \big]_{22}}.\\
\end{aligned}
\right.
\end{equation*}
The explicit values   given in Lemma \ref{tjtk} complete the proof.
\end{proof}

Notice that in Lemma \ref{preliminary_coeff}, $\overline{E}_{N-1,1}(\xi)$ appears as the denominator of the coefficients. 	In view \eqref{ejftildejsums} and using that 
$\overline{E}_{N-1,1}(\xi)$  is non-vanishing on the real line, by Wiener's Theorem in \cite[Theorem 18.21]{rudin}, we immediately get that
\begin{equation*}
\frac{1}{\overline{E}_{N-1,1}(\xi)}= \sum_{n_{\square} \in \mathbb{Z}} c_{n_{2},\dots,n_{N-1}} \mathrm{e}^{- 2 \mathrm{i} \xi l (n_{2} a_{2} + \dots n_{N-1} a_{N-1})}.
\end{equation*}
 In what follows we prove that  the infinite sum is indexed only by nonnegative integers. This will be crucial in the proof of the Theorem \ref{piecewise theorem}.
\begin{lemma}\label{denominator_expansion}
There exist constants $(c_{n_{\square} })_{n_{\square} \geq 0}$ such that
\begin{equation*}
\frac{1}{\overline{E}_{N-1,1}(\xi)}= \sum_{n_{\square} \geq 0} c_{n_{2},\dots,n_{N-1}} \mathrm{e}^{- 2 \mathrm{i} \xi l (n_{2} a_{2} + \dots n_{N-1} a_{N-1})}.
\end{equation*}
\end{lemma}

\begin{proof}
%
We will prove that all the functions  ${E}_{j,1}(\xi)$, $j=1,\dots, N-1$ can be represented as follows:
\begin{equation}
\label{rep.E}
  \frac 1{{E}_{j,1}(\xi)}=\sum_{n_{\square} \geq 0} c_{n_{2},\dots,n_{j}} \mathrm{e}^{2 \mathrm{i} \xi l (n_{2} a_{2} + \dots n_{j} a_{j})}.
\end{equation}
We use \eqref{ejftildej}, for $k=1$ and $j=\overline{1,N-1}$ to obtain the following recurrences 
\begin{equation}\label{ej1ftildej1-2}
 \left\{
  \begin{aligned}
& E_{j,1}(\xi) =  E_{j-1,1}(\xi) + \gamma_j \mathrm{e}^{2 \mathrm{i} \xi l (a_{2} + \dots + a_{j})} \widetilde{F}_{j-1,1}(\xi) \\
& \widetilde{F}_{j,1}(\xi) =  \widetilde{F}_{j-1,1}(\xi) + \gamma_j \mathrm{e}^{-2 \mathrm{i} \xi l (a_{2} + \dots + a_{j})} E_{j-1,1}(\xi)
\end{aligned}
\right.
\end{equation}
  with 
$
E_{1,1}(\xi)=1$, $ \widetilde{F}_{1,1}(\xi)=\gamma_1$,
 where $\gamma_1$ is given in \eqref{deltaepsilongamma}.

We divide the proof in two steps.

\medskip

\underline{Step 1}: We prove that for any $j=1,\dots, N-1$, $|\widetilde F_{j,1}(\xi)/E_{j,1}(\xi)|<1$. Defining $w_j=\widetilde F_{j,1}(\xi)/E_{j,1}(\xi)$ we obtain that it verifies
\[
w_j=\frac{w_{j-1}+b_j}{1+\overline{b_j}w_{j-1}}, j\geq 2,
\]
where $b_j=\gamma_j \mathrm{e}^{-2i\xi l (a_2+\dots+a_j)}$ and $w_1=\gamma_1$. Since $|b_j|<1$, the map $z\to \frac{z+b_j}{1+\overline{b_j}z}$ maps the complex unit disk $|z|<1$ to itself. Using that $|w_1|<1$ and an inductive argument  we obtain that $|w_j|<1$ for all $j\geq 2$ so $|\widetilde F_{j,1}(\xi)/E_{j,1}(\xi)|<1$ for $j=1,\dots, N-1$.
\medskip

\underline{Step 2}: We prove identity \eqref{rep.E}.  We first recall representation \eqref{ejftildejsums} of $\widetilde F_{j,1}$
\[
 \widetilde{F}_{j,1}(\xi) =  \sum_{i_{\square} \in \{ 0,1 \}} \tilde{c}_{i_{2},\dots,i_j} \mathrm{e}^{- 2 \mathrm{i} \xi l (i_{2} a_{2} + \dots+ i_j a_j)}. 
\]
It follows that the following product
\[
{e}^{2 i\xi l (a_2 + \dots + a_j)} \widetilde{F}_{j,1}(\xi)= \sum_{i_{\square} \in \{ 0,1 \}} \tilde{c}_{i_{2},\dots,i_j} \mathrm{e}^{ 2 \mathrm{i} \xi l ((1-i_{2}) a_{2} + \dots+ (1-i_j) a_j)}
\]
has a representation of the form \eqref{rep.E}. 

Let us now prove \eqref{rep.E} by an inductive argument. For $j=1$ it is obvious. Let as assume that the representation is true for $j-1$ and prove it for $j\geq 2$. The recurrences in \eqref{ej1ftildej1-2} give us that
\[
\frac 1{E_{j,1}(\xi)}=\frac{1}{E_{j-1,1}(\xi)}\frac 1{1+\gamma_j \mathrm{e}^{2i\xi l (a_2+\dots +a_j)} \widetilde F_{j-1,1}(\xi)/E_{j-1,1}(\xi)}. 
\] 
Since $1/E_{j-1,1}(\xi)$ admits representation \eqref{rep.E} it is sufficient to analyse the second factor in the above identity.
 Using Step 1 we can write 
 \[
 \frac 1{1+\gamma_j \mathrm{e}^{2i\xi la_j} \mathrm{e}^{2i\xi l (a_2+\dots +a_{j-1})}\widetilde F_{j-1,1}(\xi)/E_{j-1,1}(\xi)}=\sum _{n\geq 0} (-\gamma_j \mathrm{e}^{2i\xi la_j})^n \Big(\frac { \mathrm{e}^{2i\xi l (a_2+\dots +a_{j-1})}\widetilde{F}_{j-1,1}(\xi)}{E_{j-1,1}(\xi)}\Big)^n.
 \]
 Since both factors $ \mathrm{e}^{2i\xi l (a_2+\dots +a_{j-1})} \widetilde F_{j-1,1}(\xi)$ and $1/E_{j-1,1}(\xi)$ admit a representation as the one in \eqref{rep.E} it follows that their  $n$-th power also has a such representation. Thus the term in the left hand side has the desired representation which finishes the proof.
\end{proof}

 For any $t,x \in \mathbb{R}$, we set
\begin{align} \label{ktht}
h_t(x) & := \dfrac{1}{2 \pi} \int_{\mathbb{R}} \mathrm{e}^{-\mathrm{i} \xi^2 t} \mathrm{e}^{\mathrm{i} x \xi} \dfrac{1}{\overline{E}_{N-1,1}(\xi)} \d \xi.
\end{align}
In particular, when $N=2$, $h_t=k_t$, $k_t$ being the classical Schr\"odinger kernel.
\begin{lemma} \label{kernel piecewise new} 
Let $N\geq 2$ and $\alpha_k$, for $k=\overline{2,N}$ as in \eqref{alpha_k coefficient}.  The kernels $p_t^{1,k}$ given in \eqref{kernel piecewise} can be expressed as
\[
 p_t^{1,1}(x,y) = a_1  k_t(a_1 x -a_1 y) -  {a_1}  \sum_{i_{\square} \in \{ 0,1 \}} \tilde{c}_{i_{2},\dots,i_{N-1}} h_t [ a_1 x +a_1 y  - 2 l (i_2 a_2 + \ldots + i_{N-1} a_{N-1}) ],
\]
\[
p_t^{1,N}(x,y) =  {a_1} \alpha_{N} h_t [a_1 x - a_N (y -(N-2)l)- l (a_2 + a_3 + \ldots + a_{N-1})]
\]
and for $2\leq k\leq N-1$, $N\geq 3$, 
\begin{align*}
 p_t^{1,k}(x,y)=& {a_1}\alpha_{k} \sum_{i_{\square} \in \{ 0,1 \}} c_{i_{k+1},\dots,i_{N-1}} h_t \Big[ a_1 x- a_k( y -(k-1)  l )- l\sum_{j=2}^ka_j - 2 l \sum _{j=k+1}^{N-1} i_ja_j \Big] \\
 &-  {a_1}\alpha_{k}\sum_{i_{\square} \in \{ 0,1 \}} \tilde{c}_{i_{k+1},\dots,i_{N-1}} h_t\Big[a_1 x -a_k ((k-1)l-y) - l\sum_{j=2}^ka_j -2l \sum_{j=k+1}^{N-1} i_{j} a_{j}  \Big].
\end{align*}
\end{lemma}

\begin{proof}
 In view of \eqref{kernel piecewise}, \eqref{ejftildejsums} and Lemma \ref{preliminary_coeff}, we can rewrite the kernel $p_t^{1,1}$ in terms of the functions $k_t$ and $h_t$ as
\begin{align*}
p_t^{1,1}(x,y) & =\int_{\mathbb{R}} \mathrm{e}^{- \mathrm{i} \xi^2 t} \big[ C^{(-)}_{11}(\xi) \mathrm{e}^{\mathrm{i} \xi (a_1 x - a_1 y)} + C^{(+)}_{11}(\xi) \mathrm{e}^{\mathrm{i} \xi (a_1 x + a_1 y)} \big] \d \xi \\
& = \dfrac{a_1}{2 \pi} \int_{\mathbb{R}} \mathrm{e}^{- \mathrm{i} \xi^2 t} \mathrm{e}^{\mathrm{i} \xi (a_1 x -a_1 y)} \d \xi - \dfrac{a_1}{2 \pi} \int_{\mathbb{R}} \mathrm{e}^{-\mathrm{i} \xi^2 t} \mathrm{e}^{\mathrm{i} \xi (a_1 x +a_1 y)}\frac{\widetilde{F}_{N-1,1}(\xi)}{\overline{E}_{N-1,1}(\xi) }\d\xi\\
& = {a_1} k_t(a_1 x -a_1 y) -  {a_1}  \sum_{i_{\square} \in \{ 0,1 \}} \tilde{c}_{i_{2},\dots,i_{N-1}} h_t ((a_1 x +a_1 y) - 2 l (i_2 a_2 + \ldots + i_{N-1} a_{N-1}) ).
\end{align*}
 In the case $2 \leq k \leq N-1$, by \eqref{kernel piecewise}, 
\begin{align*}
p_t^{1,k}(x,y) & =\int_{\mathbb{R}} \mathrm{e}^{- \mathrm{i} \xi^2 t} \big[ C^{(-)}_{1k}(\xi) \mathrm{e}^{\mathrm{i} \xi (a_1 x - a_k y)} + C^{(+)}_{1k}(\xi) \mathrm{e}^{\mathrm{i} \xi (a_1 x + a_k y)} \big] \d \xi = : I^{(-)}_{1,k}(t; x,y) + I^{(+)}_{1,k}(t; x,y).
\end{align*}
Let us first remark that 
\[
\lambda_1(\xi)\dots  \lambda_{k-1}(\xi)=\exp\Big({i\xi (-a_2\ldots-a_{k-1} +(k-2)a_k)}\Big), \ k\geq 2.
\]
 Let us treat first the integral $I^{(-)}_{1,k}(t; x,y)$. By Lemma \ref{preliminary_coeff} and \eqref{lambdamu}, we can rewrite it as 
\begin{align*}
& I^{(-)}_{1,k}(t; x,y) = \alpha_{k}   \int_{\mathbb{R}} \mathrm{e}^{- \mathrm{i} \xi^2 t} \mathrm{e}^{\mathrm{i} \xi (a_1 x -a_k y)}\lambda_1(\xi)\dots \lambda_{k-1}(\xi)\frac{\overline{E}_{N-1,k}(\xi)}{\overline{E}_{N-1,1}(\xi)}\d\xi 
\\
& =  {a_1}\alpha_{k} \sum_{i_{\square} \in \{ 0,1 \}} c_{i_{k+1},\dots,i_{N-1}} h_t \Big[ a_1 x- a_k y - l \sum _{j=2}^ka_j + (k-1) a_k l - 2 l \sum _{j=k+1}^{N-1} i_ja_j \Big].
\end{align*}
 In the case of $I^{(+)}_{1,k}(t; x,y)$, we similarly arrive to
\begin{align*}
& I^{(+)}_{1,k}(t; x,y) = - a_1 \alpha_{k} \int_{\mathbb{R}} \mathrm{e}^{- \mathrm{i} \xi^2 t} \mathrm{e}^{\mathrm{i} \xi a_1 x + a_k y} 
\lambda_1(\xi)\dots \lambda_{k-1}(\xi)\mathrm{e}^{- 2 \mathrm{i} \xi l a_k (k-1)}\frac{\widetilde{F}_{N-1,k}(\xi)}{\overline{E}_{N-1,1}(\xi)}\d\xi 
\\
& =  -  {a_1}\alpha_{k}\sum_{i_{\square} \in \{ 0,1 \}} \tilde{c}_{i_{k+1},\dots,i_{N-1}} h_t\Big[a_1 x +a_k y - l \sum_{j=2}^ka_j - (k-1) a_k l -2l \sum_{j=k+1}^{N-1} i_{j} a_{j}  \Big].
\end{align*}
Thus, for all $2 \leq k \leq N-1$ we can write $p_t^{1,k}$ as
\begin{align*}
 p_t^{1,k}(x,y) =& {a_1}\alpha_{k} \sum_{i_{\square} \in \{ 0,1 \}} c_{i_{k+1},\dots,i_{N-1}} h_t \Big[ a_1 x- a_k y - l \sum _{j=2}^ka_j + (k-1) a_k l - 2 l \sum _{j=k+1}^{N-1} i_ja_j \Big] \\
& -  {a_1}\alpha_{k}\sum_{i_{\square} \in \{ 0,1 \}} \tilde{c}_{i_{k+1},\dots,i_{N-1}} h_t\Big[a_1 x +a_k y - l \sum_{j=2}^ka_j - (k-1) a_k l -2l \sum_{j=k+1}^{N-1} i_{j} a_{j}  \Big]. 
\end{align*}
In the case of $p_t^{1,N}$, we have again by \eqref{kernel piecewise} and Lemma \ref{preliminary_coeff}
\begin{align*}
 p_t^{1,N}(x,y)  &=\int_{\mathbb{R}} \mathrm{e}^{- \mathrm{i} \xi^2 t} \big[ C^{-}_{1N}(\xi) \mathrm{e}^{\mathrm{i} \xi (a_1 x - a_N y)} + C^{+}_{1N}(\xi) \mathrm{e}^{\mathrm{i} \xi (a_1 x + a_N y)} \big] \, \d \xi \\
& = {a_1}\alpha_{N}   \int_{\mathbb{R}} \mathrm{e}^{- \mathrm{i} \xi^2 t} \mathrm{e}^{\mathrm{i} \xi (a_1 x -a_N y)} \lambda_1(\xi)\ldots \lambda_{N-1}(\xi)\frac{\d\xi}{E_{N-1,1}(\xi)}\\
& =  {a_1}\alpha_{N} h_t \Big[(a_1 x- a_N y) - l (a_2 + a_3 + \ldots + a_{N-1}) + (N-2) a_N l \Big].
\end{align*}
This finishes the proof.
\end{proof}
Based on these new representations in Lemma \ref{kernel piecewise new} of the kernels in \eqref{kernel piecewise}, we can rewrite the solution $u$ expressed in \eqref{solution piecewise} in a more useful way.
\begin{lemma} \label{solution via just kt}
Let $u_0 \in L^2(\mathbb{R})$. There exists a function $\psi$, depending on $u_0$ and $(a_i)_{i=1}^N$, supported in $(0,\infty)$ such that for $t \neq 0$ and $x\leq0$, the solution of \eqref{piecewise} can be written as
\begin{equation}
\label{rep.u.psi}
u(t,x) =   \int_{\mathbb{R}} k_t(a_1x-y) u_0 \Bigg( \frac{y}{a_1} \Bigg) \mathbbm{1}_{(-\infty,0)}(y) \d y +  \int_{\mathbb{R}} h_t(a_1x-y) \psi(y) \d y.
\end{equation}
\end{lemma}

\begin{proof} We use  \eqref{solution piecewise} and   
Lemma \ref{kernel piecewise new}.
%
Using the first term in the representation of $p^{1,1}_t$ in  Lemma \ref{kernel piecewise new} and a change of variables
$y\rightarrow y/a_1$ we obtain the first term in the right hand side of \eqref{rep.u.psi}:
\[
a_1\int_{-\infty}^0 k_t(a_1x-a_1y)u_0(y)\d y=\int_{-\infty}^0 k_t(a_1x-y) u_0 \Bigg( \frac{y}{a_1} \Bigg)\d y=\int_{\mathbb{R}} k_t(a_1x-y) u_0 \Bigg( \frac{y}{a_1} \Bigg) \mathbbm{1}_{(-\infty,0)}(y) \d y.
\]
We will prove now that all the other terms in the representation of $u$ are of the form $(h_t\ast \psi )(a_1x)$ for some function $\psi$ having the support in $(0,\infty)$. 

We remark that for any $1\leq k\leq N-1$ and $y\in I_k$ the new variable
\[
z=a_k ((k-1)l-y) + l\sum_{j=2}^ka_j +2l \sum_{j=k+1}^{N-1} i_{j} a_{j}
\]
runs over the positive real numbers. We use this change of variables to obtain the existence of a function $\psi$ depending on $u_0$ and all the parameters involved in the definition of variable $z$ to obtain that 
\begin{align*}
  \int_{I_k} & h_t\Big[a_1 x -a_k ((k-1)l-y) - l\sum_{j=2}^ka_j -2l \sum_{j=k+1}^{N-1} i_{j} a_{j}  \Big]u_0(y)\d y =  \int_ {0}^\infty h_t(a_1x-z)\psi(z)\d z.
\end{align*}
 We proceed in the same way with the terms containing $a_1x-a_ky$. For any $2\leq k\leq N$ and $y\in I_k$ the variable 
 \[
 z=a_k( y -(k-1)  l )+ l\sum_{j=2}^ka_j  + 2 l \sum _{j=k+1}^{N-1} i_ja_j
 \]
 runs over positive real numbers. Thus there exists a function $\psi$ such that
 \begin{align*}
\int _{I_k} &h_t \Big[ a_1 x- a_k( y -(k-1)  l )- l\sum_{j=2}^ka_j - 2 l \sum _{j=k+1}^{N-1} i_ja_j \Big] u_0(y)\d y= \int_ {0}^\infty h_t(a_1x-z)\psi(z)\d z.
\end{align*}
In all the cases we obtain that the integrals are of the form $(h_t\ast \psi)(a_1x)$ for some function $\psi$ supported on the positive axis. Summing all these functions we obtain the desired representation for the solution $u$.
\end{proof}

\begin{lemma} \label{vanishing on negatives}
Let  $u$ be a solution of \eqref{eq.sigma}, such that
\begin{equation*}
|u(0,x)|={O}(\mathrm{e}^{-\alpha x^2}), \quad |u(1,x)|={O}(\mathrm{e}^{-\beta x^2}), \text{ as } x \to -\infty,
\end{equation*}
for some $\alpha, \beta>0$ with $\sqrt{\alpha \beta} > a_1^2/4$. Then, 
$
u(t,x)=0$
for all $ t \in \mathbb{R}$ and $ x \leq 0.
$
\end{lemma}

\begin{proof}
We use the representation in   Lemma \ref{denominator_expansion}, to write $h_t(x)$ as
\begin{equation} \label{longht}
\begin{aligned}
h_t(x) & = \dfrac{1}{ {2 \pi}} \int_{\mathbb{R}} \mathrm{e}^{-\mathrm{i} \xi^2 t} 
\mathrm{e}^{\mathrm{i} x \xi} \sum_{n_{\square} \geq 0} c_{n_{2},
 \dots,n_{N-1}} \mathrm{e}^{- 2 \mathrm{i} \xi l (n_{2} a_{2} + \ldots  
 + n_{N-1} a_{N-1})} \d \xi \\
 & = \sum_{n_{\square} \geq 0} c_{n_{2},
 \dots,n_{N-1}} k_t [ x -2l (n_{2} a_{2} + \ldots  
 + n_{N-1} a_{N-1}) ]
\end{aligned}
\end{equation}
 Using \eqref{rep.u.psi}, we have for $t\neq 0$ and $x \leq 0$, that $u(t,x)=(k_t\ast \eta)(a_1x)$
 with
\begin{align} \label{eta}
\eta(y)= u_0 \Bigg( \frac{y}{a_1} \Bigg) \mathbbm{1}_{(-\infty,0)}(y) +  \sum_{n_{\square} \geq 0} c_{n_{2},\dots,n_{N-1}}  \psi(y- 2 l (n_2 a_2+\ldots+n_{N-1}a_{N-1})) , \quad y \in \mathbb{R}.
\end{align}
Using the explicit representation of $k_t$ we have
\begin{align*}
u(t,x) & = \dfrac{1}{\sqrt{4 \pi \mathrm{i} t}} \mathrm{e}^{ \mathrm{i} \frac{a_1^2 x^2}{4t}} \int_{\mathbb{R}} \mathrm{e}^{- \mathrm{i} \frac{a_1 x y}{2 t}} \mathrm{e}^{\mathrm{i} \frac{y^2}{4 t}} \eta(y) \d y  = \dfrac{1}{\sqrt{2  \mathrm{i} t}} \mathrm{e}^{ \mathrm{i} \frac{a_1^2 x^2}{4t}} \widehat{ \mathrm{e}^{\mathrm{i} \frac{|\cdot|^2}{4 t}} \eta} \Bigg( \frac{a_1 x}{2 t} \Bigg).
\end{align*}
Since $|u(1,x)|={O}(\mathrm{e}^{-\beta x^2})$ as $x\rightarrow -\infty$ we obtain 
\begin{equation*}
| \widehat{ \mathrm{e}^{\mathrm{i} \frac{|\cdot|^2}{4 }} \eta} (x) |= {O}(\mathrm{e}^{-  \frac{4\beta x^2}{ a_1^2}}), \quad \mbox{as}\ x\rightarrow -\infty.
\end{equation*}
 Since $ \mathrm{supp} (\psi) \subseteq (0, \infty) $ and  $a_2,  \ldots, a_{N-1} >0$, we have for any $y \leq 0$ that $\eta(y)=u_0(y/a_1)$. The property $|u(0,x)|={O}(\mathrm{e}^{-\alpha x^2})$ as $x\rightarrow-\infty$ gives us that 
\[
 |\mathrm{e}^{\mathrm{i} \frac{|\cdot|^2}{4  }} \eta(x) | =  {O}(\mathrm{e}^{- \alpha \frac{x^2}{a_1^2}}),\quad \mbox{as}\ x\rightarrow -\infty.
\]
 Thus, by \cite[Theorem 2.3 (B)]{nazarov}, it follows that as long as $\sqrt{ \alpha \beta}> a_1^2/4 $ with $\alpha, \beta >0$, we must have $\eta \equiv 0 $ on $\mathbb{R}$,
 which implies $u(t,x)=0$, for $t\neq 0$ and $x \leq 0$, which completes the proof.
\end{proof}

We are ready to prove the main result of this section, Theorem \ref{piecewise}.

\begin{proof}[Proof of Theorem \ref{piecewise}] 
We prove the first part since the other two follow from the first one.
We will proceed by induction. For $N\geq 1$ let $P(N)$ be the statement: For any $a_1, \ldots, a_N >0$, if the solution $u_{\sigma_N}$ of the equation
\begin{equation*}
\left\{
\begin{array}{ll}
\mathrm{i} u _t(t,x)+ \partial_x (\sigma_N \partial_x u)(t,x)=0,&  x\in \mathbb{R}, t\neq 0 ,\\
u(0,x)=u_0(x),&   x\in \mathbb{R}
\end{array}
\right.
\end{equation*}
with the piecewise constant function $\sigma_N$ given by $\sigma_1(x)=a_1^{-2}$ if $N=1$ and for $N\geq 2$
$\sigma_N(x)=a_k^{-2}$, $x\in I_k$, $k=1,\dots, N$,
satisfies for some positive numbers $\alpha$ and $\beta$ with $\alpha\beta > a_1^4/16$,
\begin{equation*} 
 u_{\sigma_N}(0,x)= {O}( \mathrm{e}^{-\alpha x^2}), \ u_{\sigma_N}(1,x) = {O}(\mathrm{e}^{-\beta x^2}), \text{ as } x \to -\infty,
\end{equation*}
 then $u_{\sigma_N}\equiv 0$. 

When $N=1$ let us consider  $u_{\sigma_1}$, solution of
\[
\mathrm{i} u _t(t,x)+ \frac{1}{a_1^2} \partial^2_{xx} u (t,x)=0,  x\in \mathbb{R}, t\neq 0
\]
 that satisfies 
$
 u(0,x) = {O}( \mathrm{e}^{-\alpha x^2})$ and $u(1,x)={O}(\mathrm{e}^{-\beta x^2})$ as $x \to -\infty$,
for some positive numbers $\alpha$ and $\beta$ with $\alpha\beta > {a_1^4}/{16}$. We consider $v(t,x)=u_{\sigma_1}(t,x/a_1)$ and apply the results for the one dimensional LSE  \cite[Theorem 2.3 (B)]{nazarov} to conclude that $u_{\sigma_1}\equiv 0$.

%
%

Assume now that $P(N)$, holds true, and we want to prove that $P(N+1)$ also holds true, i.e. we want to show that for any $a_1, \ldots, a_{N+1} >0$ and  $ \sigma_{N+1}(x)=a_k^{-2}$, $x\in I_k$, $k=1,\dots, N+1$, 
the solution $u_{\sigma_{N+1}}$ of the equation
\[
\mathrm{i} u _t(t,x)+ \partial_x (\sigma_{N+1} \partial_x u)(t,x)=0,\  x\in \mathbb{R}, t\neq 0
\]
satisfying for some positive numbers $\alpha$ and $\beta$ with $\alpha\beta > a_1^4/16$
\begin{equation*} 
 u_{\sigma_{N+1}}(0,x)={O}( \mathrm{e}^{-\alpha x^2}), \ u_{\sigma_{N+1}}(1,x) = {O}(\mathrm{e}^{-\beta x^2}),  \text{ as } x \to -\infty,
\end{equation*}
vanishes identically,  $u_{\sigma_{N+1}}\equiv 0$. 

 Fix $a_1, \ldots, a_N, a_{N+1} >0$ and consider the corresponding piecewise constant function $\sigma_{N+1}$. Then, by Lemma \ref{vanishing on negatives} applied for $\sigma_{N+1}$, since $\alpha \beta > a_1^4 /16$, the solution $u_{\sigma_{N+1}}$ vanishes for all $t \in \mathbb{R}$ and $x \leq 0$. Then, one can check that $u_{\sigma_{N+1}}$ is solution to
\begin{equation*}
\left\{
\begin{array}{ll}
\mathrm{i} u _t(t,x)+ \partial_x (\widetilde{\sigma}_{N} \partial_x u)(t,x)=0,&  x\in \mathbb{R}, t\neq 0 ,\\
u(0)=u_0,&   x\in \mathbb{R}
\end{array}
\right.
\end{equation*}
with $\widetilde \sigma_N=a_2^{-2}$  when $N=1$ and  for $N\geq 2$
\begin{equation*}
 \widetilde{\sigma}_N(x)=
 \left\{
  \begin{aligned}
  	a_2^{-2}, & \quad x \in \tilde  I_1:=(-\infty, l),\\
  	a_3^{-2}, & \quad x \in\tilde I_2:=(l,2 l),\\
  	\vdots \\
  	a_{N+1}^{-2}, & \quad x \in\tilde I_N:=((N-1)l,+\infty).\\
  \end{aligned}
\right.
\end{equation*}
The translated function $v_{\sigma_N}(t,x):=u_{\sigma_{N+1}}(t,x+l)$ is solution of
\begin{equation}
\left\{
\begin{array}{ll}
\mathrm{i} v _t(t,x)+ \partial_x (\sigma_{N} \partial_x v)(t,x)=0,&  x\in \mathbb{R}, t\neq 0 ,\\
v(0,x)=u_0(x+l),&   x\in \mathbb{R}
\end{array}
\right.
\end{equation}
 with $ \sigma_N(x)=a_{k+1}^{-2}$, $x\in I_k$, $k=1,\dots, N$.
Since $u_{\sigma_{N+1}}$ vanishes for all $t \in \mathbb{R}$ and $x \leq 0$ function $v$ satisfies
\begin{equation*} 
 v_{\sigma_{N}}(0,x)=0, \ v_{\sigma_{N}}(1,x) =0,  \text{ for } x < -l.
\end{equation*}
Thus
\begin{equation*}
v_{\sigma_{N}}(0,x) ={O}(\mathrm{e}^{-\alpha x^2}), \ v_{\sigma_{N}}(1,x) = {O}(\mathrm{e}^{-\beta x^2}), \text{ as } x \to -\infty,
\end{equation*}
for any $\alpha$ and $\beta$ satisfying $\alpha\beta >0$, in particular, for some positive numbers $\alpha$ and $\beta$ satisfying $\alpha\beta > a_2^4/16$ and thus, by the induction hypothesis, $v_{\sigma_{N}}\equiv 0$. Finally, we get that $u_{\sigma_{N+1}} \equiv 0$ and, therefore, $P(N+1)$ also holds true. 

In order to complete the proof of Theorem \ref{piecewise theorem}, let us show that in the case of two steps piecewise-constant function $\sigma$, i.e. $N=2$, the exponents are sharp. More precisely, let
\begin{equation*}\label{sigma2piecewise}
 \sigma(x)=
 \left\{
  \begin{aligned}
  	a_1^{-2}, & \quad x \in I_1:=(-\infty,0),\\
  	a_2^{-2}, & \quad x \in I_2:=(0,\infty),\\
  \end{aligned}
\right.
\end{equation*}
with $a_1, a_2>0$. We note that when $N=2$  in view of \eqref{ktht} $h_t=k_t$. Using the representation of $u$ above (see also \cite{gaveauokada87}), the solution $u$ of system \eqref{eq.sigma} with $\sigma$ as above, can be written as
\begin{equation*}
u(t,x)=
\begin{cases}
(k_t\ast  \psi)(a_1 x), & x <0, \\
(k_t\ast  \widetilde\psi)(a_2 x), & x >0,
\end{cases}
\end{equation*} 
 with
\begin{equation*}
\begin{cases}
\psi(y)=u_0 \bigg( \dfrac{y}{a_1} \bigg) \mathbbm{1}_{(-\infty,0)}(y)+ \dfrac{a_2-a_1}{a_1+a_2}  u_0\Big(-\dfrac{y}{a_1} \Big) \mathbbm{1}_{(0,\infty)}(y)+ \dfrac{2 a_1}{a_1+a_2} u_0 \bigg( \dfrac{y}{a_2} \bigg) \mathbbm{1}_{(0, \infty)}(y), \\
\widetilde{\psi}(y)= \dfrac{2 a_2}{a_1+a_2} u_0 \bigg( \dfrac{y}{a_1} \bigg) \mathbbm{1}_{(-\infty,0)}(y) + u_0 \bigg( \dfrac{y}{a_2} \bigg) \mathbbm{1}_{(0,\infty)}(y) + \dfrac{a_1-a_2}{a_1+a_2} u_0 \bigg(-\dfrac{y}{a_2} \bigg) \mathbbm{1}_{(-\infty,0)}(y).
\end{cases}
\end{equation*}

 Taking as initial data
\begin{equation*}
u_0(x)=
\begin{cases}
\mathrm{e}^{-  a_1^2 x^2 - \mathrm{i} a_1^2 \frac{x^2}{4}}, & x \leq 0 \\
\mathrm{e}^{-  a_2^2 x^2 - \mathrm{i} a_2^2 \frac{x^2}{4}}, & x > 0
\end{cases},
\end{equation*}
the solution at $t=1$ can be written as
\begin{equation*}
u(1,x)=
\begin{cases}
\sqrt{\frac{2}{\mathrm{i}}} \mathrm{e}^{\mathrm{i} a_1^2 \frac{x^2}{4}} \widehat{\mathrm{e}^{- |\cdot|^2}} \big( \frac{2 a_1 x}{4} \big) , & x \leq 0 \\
\sqrt{\frac{2}{\mathrm{i}}} \mathrm{e}^{\mathrm{i} a_2^2 \frac{x^2}{4}} \widehat{\mathrm{e}^{-  |\cdot|^2}} \big( \frac{2 a_2 x}{4} \big), & x > 0
\end{cases}
\quad
=
\quad
\begin{cases}
\sqrt{\frac{4  }{\mathrm{i}}}  \mathrm{e}^{- \frac{a_1^2 x^2}{16  }+\mathrm{i} a_1^2 \frac{x^2}{4}}, & x \leq 0 \\
\sqrt{\frac{4  }{\mathrm{i}}}  \mathrm{e}^{- \frac{a_2^2 x^2}{16 }+\mathrm{i} a_2^2 \frac{x^2}{4}}, & x > 0
\end{cases},
\end{equation*}
and letting $\alpha =   \min \{ a_1^2, a_2^2\}$, one gets a nonzero solution  satisfying
\begin{equation*}
\ |u(0,x)|\lesssim \mathrm{e}^{-\alpha x^2}, \ |u(1,x)|\lesssim \mathrm{e}^{- \frac{\alpha}{16 } x^2}, \quad \text{ as } |x| \to \infty.
\end{equation*}
The proof is now complete.
\end{proof}


\section{A Carleman Inequality}\label{carleman-tree}
Let $\Gamma$ be a star-shaped graph, $N$ the number of its edges and the critical exponent $\gamma_\Gamma$ as in \eqref{gamma}. In the following, we will obtain a Carleman inequality on $\Gamma$ on which we rely the proof of Theorem \ref{main-potential}. Let consider the set $\mathcal{Z}_{comp}$ defined by
\begin{align*}
\label{}
  \mathcal{Z}_{comp}=\Big\{{\bf{q}}=(q_j)_{j=\overline{1,N}}  \in C([0,T]\times \Gamma), \ q_j \in C^{1,2} ([0,T] \times [0, \infty)) \ \forall j=\overline{1,N} \ s.t.\\
  q_j(t,0)=q_l(t,0) \quad  \forall 1 \leq j,l \leq N, \quad 
\sum_{j=1}^N q_{j,x}(t,0)=0, \quad t \in [0,T] \quad \Big\}
\end{align*}
%
%
It is clear that $\mathcal{Z}_{comp}$ is densely embedded in $ C([0,T],D(\Delta_\Gamma))$.
Consider also the weight function $\bm{\varphi}=(\varphi_j)_{j=\overline{1,N}}$ given by
\begin{equation}
\varphi_j(t,x)=\mu |\alpha_j x + R t(1-t) |^2 - \frac{(1+\epsilon)R^2 t(1-t)}{16 \mu} \quad \forall j=\overline{1,N},
\end{equation}
for some $\mu>0, \epsilon>0, R>0$ and $\alpha_j \in \mathbb{R}$ for all $j \in \overline{1,N}$, such that
\begin{equation}\label{sum.derivatives.alpha}
\sum_{j=1}^N \varphi_{j,x} (t,0)=0,
\end{equation}
i.e., in terms of the vector $\bm{\alpha}=(\alpha_j)_{j=\overline{1,N}}$,
$
\sum_{j=1}^N \alpha_j=0.
$
Moreover, one can observe that
$
\varphi_j(0,t)=\varphi_l(0,t),
$
for any $1 \leq j, l, \leq N$ and $t \in [0,T]$, so the weight function $\bm{\varphi}$ belongs to $\mathcal{Z}_{comp}$.

In the proof of the Carleman inequality, we will make use of $N$ weights $(\bm{\varphi}^k)_{k=\overline{1,N}}$, with coefficients $(\bm{\alpha}^k)_{k=\overline{1,N}}$ such that

\begin{enumerate}[label=(\roman*)]

\item If $N$ is even, $\bm{\alpha}^1=(1,-1,\dots,1,-1)
$ and $\bm{\alpha}^k$ \text{ is a cyclic permutation of }  $\bm{\alpha}^{k-1}$, for all  $k= 2,\dots, N$.
%
\item If $N=2m+1$, $\bm{\alpha}^1=(\underbrace{-1,\dots,-1}_{m+1}, \underbrace{\frac{m+1}m,\dots, \frac{m+1}m}_{m})$ and $\bm{\alpha}^k$ \text{ is a cyclic permutation of }  $\bm{\alpha}^{k-1}$, for all  $k= 2,\dots, N$.

%
\end{enumerate}
The particular form the of the vectors $\bm{\alpha}^k$ satisfies the following  properties that will be used in the proof of a Carleman inequality:
\[
\sum_{k}\alpha_j^k=0, \quad\forall j=1,\dots, N,
\]
and the sum
$
\sum_{k}(\alpha_j^k)^2 $ \text{is independent on} $j$.

\begin{lemma}(Carleman Inequality)
	\label{lemma-carleman}
	Let us consider the weights introduced above. The following inequality 
	\begin{equation}
\label{ineg-carleman}
\frac{R^2\epsilon}{8\mu}  \sum _{k=1}^N \|\mathrm{e}^{\bm{\varphi} ^k} \bm{q}\|^2_{L^2([0,1]\times \Gamma)}\leq \sum _{k=1}^N \| \mathrm{e}^{\bm{\varphi} ^k} (\partial _t +\mathrm{i} \Delta_\Gamma)\bm{q}\|^2_{L^2([0,1]\times \Gamma)}
\end{equation}
holds for all $\epsilon>0$, $\mu>0$, $R>0$ and $\bm{q}\in \mathcal{Z}_{comp}$.
\end{lemma}

\begin{proof}
	Writing explicitly the above norms we will prove that 
	\[
	\frac{R^2\epsilon}{8\mu}  \sum _{k}\sum _{j=1}^N \int_0^1\int_0^\infty |\mathrm{e}^{\varphi_j ^k(x,t)} q_j(t,x)|^2 \d x\d t \leq \sum _{k} \sum _{j=1}^N \int_0^1\int_0^\infty | \mathrm{e}^{\varphi_j ^k(x,t)} (\partial _t +\mathrm{i} \partial_{xx})q_j(t,x)|^2 \d x\d t.
	\]
	
In the following, for the sake of reading we will not make precise the time dependence unless it si necessary.   
Following \cite{ignatpazotorosier}, for each $k \in \{1, \dots, N \}$, we denote  $\bu^k:=\mathrm{e}^{\bm{\varphi}^k} \bm{q},$
and
\begin{equation}
\bm{w}^k:=\mathrm{e}^{\bm{\varphi}^k} (\partial_t + \mathrm{i} \Delta_{\Gamma}) \bm{q}= \mathrm{e}^{\bm{\varphi}^k} (\partial_t + \mathrm{i} \Delta_{\Gamma}) \mathrm{e}^{-\bm{\varphi}^k} \bm{u}^k.
\end{equation}
Then,
\begin{align}
\sum_{k} \| \bm{w}^k  \|_{L^2([0,1]\times \Gamma)} \geq \sum_{k} \sum_{j=1}^N 4 \int_0^1\int_0^{\infty} \varphi^k_{j,xx} | u^k_{j,x} |^2 \d x\d t + 4 \Im\Big(\int_0^1 \int_0^{\infty} \varphi^k_{j,xt} u^k_j \overline{u}^k_{j,x} \d x\d t \Big) \\
+ \int_0^1 \int_0^{\infty} |u^k_j|^2 [  -(\varphi^k_{j,4x} - \varphi^k_{j,tt}) + 4 (\varphi^k_{j,x})^2 \varphi^k_{j,xx} ] \d x\d t + BT(0),
\end{align}
where the boundary term at $x=0$ is given by 
\begin{align}
BT(0) &=2 \sum_{k} \sum_{j=1}^{N} \int_0^1 \varphi^k_{j,x}(0) | u^k_{j,x} (0) |^2  \d t + \sum_{k} \sum_{j=1}^{N} \int_0^1 (-\varphi^k_{j,3x}(0)+ 2 (\varphi^k_{j,x}(0)^3))|\bu(0)|^2 \d t \\
&+ 2 \sum_{k} \sum_{j=1}^{N} \int_0^1 \varphi^k_{j,xx}(0) \Re(\bu(0) \overline{u}^k_{j,x}(0)) \d t+ 2 \sum_{k} \sum_{j=1}^{N} \int_0^1 \varphi^k_{j,t}(0) \Re(-\mathrm{i} \bu(0) \overline{\bu}^k_{j,x}(0)) \d t\\
&+ \sum_{k} \sum_{j=1}^{N} \int_0^1 \mathrm{i} \varphi^k_{j,x}(0) [ \bu(0) \overline{\bu}_t(0) - {\bu}_t(0) \overline{\bu}(0)] \d t=: J_1+J_2+J_3+J_4+J_5,
\end{align}
with $\bu(0)=u^k_j(0,t)$ and $\bm{\varphi}(0)=\varphi^k_j(0,t),
$
and similarly for the times derivatives.

In view of property \eqref{sum.derivatives.alpha} we immediatelly obtain that $J_5=0$. We now proceed with the other terms.
 For the first one we have
\begin{align}
J_1 =& 2 \sum_{k} \sum_{j=1}^N \int_0^1 \varphi^k_{j,x}(0) | \varphi^k_{j,x} (0) q_j(0) + q_{j,x}(0) |^2 \mathrm{e}^{2 \varphi^k_j (0)} \d t \\
=& 2 \sum_{k} \sum_{j=1}^N \int_0^1 (\varphi^k_{j,x}(0))^3 | \bu(0) |^2 \d t + 2 \sum_{k} \sum_{j=1}^N \int_0^1 \varphi^k_{j,x}(0) \mathrm{e}^{2\varphi^k_{j}(0)} |q_{j,x}(0)|^2 \d t \\
&+ 4 \sum_{k} \sum_{j=1}^N \int_0^1 \varphi^k_{j,x}(0) \Re(\varphi^k_{j,x}(0)q_j(0)\overline{q}_{j,x}(0)) \mathrm{e}^{2 \varphi (0)} \d t \\
=& 2\int_0^1  | \bu(0) |^2   \sum_{k} \sum_{j=1}^N (\varphi^k_{j,x}(0))^3\d t + 2 \int_0^1 \mathrm{e}^{2\bm{\varphi} (0)} \sum_{k} \sum_{j=1}^N  \varphi^k_{j,x}(0) |q_{j,x}(0)|^2 \d t \\
&+ 4 \Re \int_0^1 \bu(0)  \mathrm{e}^{\bm{\varphi} (0)}\sum_{j=1}^N   \overline{q}_{j,x}(0)  \sum_{k} (\varphi^k_{j,x}(0))^2 \d t.
\end{align}
We prove that the last two term vanish. Using that
$
\sum_{k} \varphi^k_{j,x}(0)=0$  for all $j=1,\dots, N$,
we obtain that the second one vanishes.
 Since ${\sum_{k}} (\varphi^k_{j,x}(0))^2$ is independent of $j$ and $\sum_{j=1}^N q_{j,x}(0)=0$, the last term in the above right hand side is zero. 
 For any $j=1,\dots,N$,  $\sum_{k} (\varphi^k_{j,x}(t,0))^3= A(t)$ where $A(t)\geq 0$. In particular, when $N$ is even $A(t)=0$. This gives us that $J_1\geq 0$.
%
%
In the case of $J_2$ we use that all the third order derivatives of $\varphi^k$ vanish and we have $J_2=J_1\geq 0$.
 In the case of $J_3$ we use that 
\begin{align}
J_3&=\sum_{k} \sum_{j=1}^N \int_0^1 \varphi^k_{j,xx}(0)  \Re[\bu(0) \mathrm{e}^{\varphi (0)}  ( \varphi^k_{j,x}(0) \overline{q}_{j}(0) +   \overline{q}_{j,x}(0) )  ]\d t = \\
&= \sum_{k} \sum_{j=1}^N \int_0^1 \varphi^k_{j,xx}(0) \Re[|\bu(0)|^2 \varphi^k_{j,x}(0) + \bu(0)  \mathrm{e}^{\varphi (0)}\overline{q}_{j,x}(0)   ]\d t.
\end{align}
 Since $\varphi^{k}_j(0)=\varphi^{\tilde{k}}_{\tilde{j}}(0)$, for any $k,j,\tilde{k},\tilde{j} \in \{1, \dots, N \}$ and $\sum_k \varphi^k_{j,xx}(0)$ does not depend on $j$, we have that the sum of the last term vanishes and 
 therefore,
\begin{equation}
\begin{aligned}
J_3 &=\int_0^1 |\bu(0)|^2  \sum_{k} \sum_{j=1}^N \varphi^k_{j,xx}(0)  \varphi^k_{j,x}(0) \d t
=4 R\mu^2 \int_0^1 |\bu(0)|^2   t(1-t)\d t  \sum_{k} \sum_{j=1}^N (\alpha^k_j)^3 \geq 0.
\end{aligned}
\end{equation}
 Denoting by $\varphi_t(t,0)$ the common value of $\varphi^k_{j,t}(t,0)$, $1\leq j,k\leq N$ we get
\begin{align*}
J_4 &=2 \sum_{k} \sum_{j=1}^N \int_0^1 \varphi_t(t,0) \Re [-\mathrm{i} \bu(0) ( \varphi^k_{j,x}(0) \overline{q}_j(0)+ \overline{q}_{j,x}(0) ) \mathrm{e}^{\varphi (0)}] \d t \\
&= 2 \sum_{k} \sum_{j=1}^N \int_0^1 \varphi_t(t,0) \Re [-\mathrm{i} \bu(0) \varphi^k_{j,x}(0) \overline{\bu}(0)- \mathrm{i} \bu(0) \overline{q}_{j,x}(0)  \mathrm{e}^{\varphi (0)}] \d t \\
&= 2 N \int_0^1 \varphi_t(t,0) \mathrm{e}^{\varphi(0)}\Re [- \mathrm{i} \bu(0) \sum_{j=1}^N\overline{q}_{j,x}(0)  ] \d t=0,
\end{align*}
 where we used the fact that $\varphi^k_{j}$ are real valued functions.
%
%
%
The above estimates show that $BT(0) \geq 0$ and therefore
\begin{align*}
\sum_k \| \bm{w}^k \|^2_{L^2([0,1]\times \Gamma)} & \geq \sum_{k} \sum_{j=1}^N \int_0^1 \int_0^{\infty} 4\varphi^k_{j,xx} |u^k_{j,x}|^2 \d x\d t + 4 \Im \int_0^1 \int_0^{\infty} \varphi^k_{j,xt} u^k_{j} \overline{u}^k_{j,x} \d x\d t \\
&+ \sum_{k} \sum_{j=1}^N \int_0^1 \int_0^{\infty} |u^k_j|^2 [ -\varphi^k_{j,4x} + \varphi^k_{j,tt}+4 (\varphi^k_{j,x})^2 \varphi^k_{j,xx} ] \d x\d t.
\end{align*}
Notice that for all $k$ and $j$ in  $\{1, \dots, N \}$, $\varphi^k_{j,4x}(x,t)=0$ and $\varphi^k_{j,xx}(x,t)=2 \mu (\alpha^k_j)^2$.
Then, we make squares and we can write
\begin{align}
\sum_k \| w^k \|^2_{L^2([0,1]\times \Gamma)} &\geq \sum_k \sum_{j=1}^{N} \int_0^1 \int_0^{\infty} \Big| 2 \sqrt{\varphi^k_{j,xx}} u^k_{j,x} - \dfrac{{\rm{i}}\varphi^k_{j,xt}}{\sqrt{\varphi^k_{j,xx}}} u^k_j \Big|^2 \d x \d t  \\
&+ \sum_k \sum_{j=1}^{N} \int_0^1 \int_0^{\infty}\Big[\varphi^k_{j,tt}+4\varphi_{j,xx}^k (\varphi^k_{j,x})^2 - \dfrac{(\varphi^k_{j,xt})^2}{ \varphi^k_{j,xx}} \Big]|u^k_j|^2 \d x \d t.
\end{align}
From the definition of $u^k_j$ we get that
\begin{align*}
\varphi^k_{j,tt}+4\varphi_{j,xx}^k &(\varphi^k_{j,x})^2 - \dfrac{(\varphi^k_{j,xt})^2}{ \varphi^k_{j,xx}} \\
&= 32 (\alpha^k_j)^4 \mu^3 (\alpha^k_j x + R t (1-t))^2 - 4 R \mu (\alpha^k_j x +R t (1-t)) + \dfrac{R^2 (1+\epsilon)}{8 \mu}\\
&= 32 (\alpha^k_j)^4 \mu^3\Big(\alpha^k_j x + R t (1-t) - \dfrac{R}{16 \mu^2 (\alpha^k_j)^4}\Big)^2 - \dfrac{R^2}{8 \mu (\alpha^k_j)^4} + \dfrac{R^2 (1 + \epsilon)}{8 \mu}\\
&\geq \dfrac{R^2}{8 \mu}\Big(1 - \dfrac{1}{(\alpha^k_j)^4} + \epsilon\Big) \geq \dfrac{R^2 \epsilon}{8 \mu},
\end{align*}
where the last inequality holds due to the fact that $|\alpha^k_j| \geq 1$, for all $k$ and $j$ in $\{1, \dots, N \}$.
Finally this implies that
\begin{equation}
\sum_k \| w^k \|^2_{L^2([0,1]\times \Gamma)} \geq \dfrac{R^2 \epsilon}{8 \mu} \sum_k \sum_{j=1}^{N} \int_0^1 \int_0^{\infty} |u^k_j|^2 \d x \d t,
\end{equation}
which concludes the result.
\end{proof}

\section{Proof of Theorem \ref{main-potential}}\label{potential-tree}

Before proving Theorem \ref{main-potential}, we need to study the behavior of a solution of the Schr\"odinger equation
\begin{align}\label{eq.treeg}
\left\{
\begin{array}{l}
\bu_t=i(\Delta_\Gamma+\bv(t,x))\bu \quad \text{in}\ [0,1]\times \Gamma,\\
\bu(0)=\bu_0,   x\in \Gamma.
\end{array}
\right.
\end{align}
in the star-shaped graph $\Gamma$ with Gaussian decay at $t=0$ and $t=1$. More precisely, we need to show that such a solution has Gaussian decay at any time in between. 

Through this section, we will denote by $\|\cdot\|_2$ and $\|\cdot\|_\infty$ the $L^2(\Gamma)$ and $L^\infty(\Gamma)$ norms.

\begin{theorem}\label{thm56}
Assume that $u$ in $C([0,1],L^2(\Gamma))$ verifies \eqref{eq.treeg}, $\bv(t,x)=\bv_1(x)+\bv_2(t,x)$ where $\bv_1$ is real-valued, $\|\bv_1\|_\infty\le M_1$ and that there are two positive numbers $\alpha$ and $\beta$ such that
\[
\|\mathrm{e}^{\alpha|x|^2}\bu(0)\|_2,\ \|\mathrm{e}^{\beta|x|^2}\bu(1)\|_2,\ \text{and }\sup_{[0,1]}\|\mathrm{e}^{\frac{\alpha\beta|x|^2}{(\sqrt{\alpha} t+(1-t)\sqrt{\beta})^2}}\bv_2(t)\|_\infty<+\infty.
\] 
Then, there is a constant $\mathcal{N}=\mathcal{N}(\alpha,\beta)$ such that
\[
\|\mathrm{e}^{\frac{\alpha\beta|x|^2}{(\sqrt{\alpha} t+(1-t)\sqrt{\beta})^2}}\bu(t)\|_2\le \mathrm{e}^{\mathcal{N}(M_1+M_2+M_1^2+M_2^2)}\|\mathrm{e}^{\alpha|x|^2}\bu(0)\|_2^{\frac{\sqrt{\beta}(1-t)}{\sqrt{\alpha} t+\sqrt{\beta}(1-t)}}\|\mathrm{e}^{\beta|x|^2}\bu(1)\|_2^{\frac{\sqrt{\alpha} t}{\sqrt{\alpha} t+\sqrt{\beta}(1-t)}},
\]
when $0\le t\le 1,\ M_2=\sup_{[0,1]}\|\mathrm{e}^{\frac{\alpha\beta|x|^2}{(\sqrt{\alpha} t+(1-t)\sqrt{\beta})^2}}\bv_2(t)\|_\infty \mathrm{e}^{2\sup_{[0,1]}\|\Im \bv_2(t)\|_\infty}$. Moreover
\[
\|\sqrt{t(1-t)}\mathrm{e}^{\frac{\alpha\beta|x|^2}{(\sqrt{\alpha} t+(1-t)\sqrt{\beta})^2}}\nabla\bu\|_{L^2([0,1]\times\Gamma)}\le \mathcal{N}\mathrm{e}^{\mathcal{N}(M_1+M_2+M_1^2+M_2^2)}\left[ \|\mathrm{e}^{\alpha|x|^2}\bu(0)\|_2+\|\mathrm{e}^{\beta|x|^2}\bu(1)\|_2 \right].
\]
\end{theorem}

The proof of this Theorem follows very closely the proof of the result in the real line, given in \cite{MR2443923}, and therefore we skip the details.

\begin{proof}[Proof of Theorem \ref{main-potential}] 
Using the Appell transform (see Section \ref{sec-Appell transform}) we can consider the case $\alpha=\beta=\gamma>2\gamma^2_\Gamma$.
%
 The subsequent formal computations are justified by Theorem \ref{thm56}.
 Since $\gamma>2\gamma_\Gamma^2$, we can choose $\mu>1/2$ and $\epsilon>0$ such that 
\begin{equation}\label{condition on mu}
\dfrac{(2 \gamma_\Gamma)^2 (1+\epsilon)^{3/2}}{2 (1- \epsilon)^3} < (2 \gamma_\Gamma)^2 \mu \leq \dfrac{\gamma}{1+\epsilon},
\end{equation}
 and the smooth functions $\theta_M$ and $\eta_R$, for $M \gg R>2$, verifying
\begin{equation}
\theta_M(x)= \begin{cases}
1, & x \in [0,M]\\
0, & x \in (2M,\infty)
\end{cases}, 0\leq \theta_M\leq 1,
\end{equation}
\begin{equation}
\eta_R(t)= \begin{cases}
1, & t \in [\frac{1}{R}, 1- \frac{1}{R}]\\
0, & t \in [0, \frac{1}{2R}] \cup [1-\frac{1}{2R}, 1]
\end{cases}, \ 0\leq \eta_R\leq 1.
\end{equation}
 We define the space-time truncation of $\bu$ 
\begin{equation}
q_j(t,x)=\theta_M(x) \eta_R(t) u_j(t,x),
\end{equation}
 and since $\bm{q}=(q_j)_{j=\overline{1,N}} \in \mathcal{Z}_{comp}$, we can use the previous Carleman estimate.
 Note also that
\begin{equation}
\begin{aligned}
(\partial_t+ \mathrm{i} \partial_{xx})q_j = \theta_M \eta_R (\partial_t + \mathrm{i} \partial_{xx}) u_j + \eta'_R \theta_M u_j + (\theta''_M \eta_R u_j + 2 \theta'_M \eta_R u_{j,x}).
\end{aligned}
\end{equation}
 Therefore, in view of the Carleman estimates \eqref{ineg-carleman} 
\begin{equation}\label{3terms}
\begin{aligned}
\frac{R^2\epsilon}{8\mu}  \sum _{k=1}^N \|\mathrm{e}^{\bm{\varphi} ^k} \bm{q}\|^2_{L^2([0,1]\times \Gamma)} \leq  & \sum _{k=1}^N \sum_{j=1}^N \int_0^1 \int_0^{\infty}  | \mathrm{e}^{\varphi^k_j} V_j q_j |^2 + | \mathrm{e}^{\varphi^k_j} \eta'_R \theta_M u_j |^2 \d x \d t \\
&+ \sum _{k=1}^N \sum_{j=1}^N \int_0^1 \int_0^{\infty} | \mathrm{e}^{\varphi^k_j} \eta_R(\theta''_M u_j + 2 \theta'_M u_{j,x}) |^2 \d x\d t.
\end{aligned}
\end{equation}
 Since 
\begin{equation}
 \sum _{k=1}^N \sum_{j=1}^N \int_0^1 \int_0^{\infty}  |\mathrm{e}^{\varphi^k_j} V_j q_j  |^2 \d x \d t \leq \|  \bm{V} \|^2_{\infty} \sum _{k=1}^N \sum_{j=1}^N \int_0^1 \int_0^{\infty} |\mathrm{e}^{\varphi^k_j} q_j |^2 \d x \d t,
\end{equation}
 taking $R \geq 2 \sqrt{\dfrac{8 \mu}{\epsilon} } \| \bm{V} \|_\infty$, this term can be absorbed in the left-hand side. In the following computations, all the constants involved may depend on the behavior of the solution at times $t=0$ and $t=1$, the potential, and also on the parameters $\gamma, \mu$ or $\epsilon$, but they will not depend on $R$ and $M$.

Since the integrand in the second term is in fact supported in $x \in [0, 2M]$ and $t \in \bigg[\dfrac{1}{2R},\dfrac{1}{R}\bigg] \bigcup \bigg[1-\dfrac{1}{R}, 1-\dfrac{1}{2R}\bigg]$ by \eqref{condition on mu} we have that for such $x$ and $t$
\begin{align}
\varphi^k_j(&t,x)  \leq \mu [(\alpha^k_j)^2 x^2 + R^2 t^2 (1-t)^2 + 2 \alpha^k_j x R t (1-t)]\\
& \leq (\alpha^k_j)^2 \mu (1+\epsilon) x^2 + R^2 t^2 (1-t)^2 \mu \bigg(1+\dfrac{1}{\epsilon}\bigg) \leq \max_{k,j} |\alpha^k_j|^2 \mu (1+\epsilon)x^2 + \dfrac{\gamma}{(2 \gamma_\Gamma)^2 \epsilon}.
\end{align}
Since $|\alpha_j^k|\leq 2\gamma_\Gamma$ the second inequality in \eqref{condition on mu} gives us that 
\[
\varphi^k_j(t,x)  \leq \gamma x^2 + \dfrac{\gamma}{(2 \gamma_\Gamma)^2 \epsilon}.
\]
%
 Hence, we can estimate all the terms uniformly in $k$ and obtain that
\begin{equation}
\begin{aligned}
 \sum _{k=1}^N\sum_{j=1}^N \int_0^1 \int_0^{\infty} | \mathrm{e}^{\varphi^k_j} \eta'_R \theta_M u_j |^2 \d x \d t &\leq N R^2 \mathrm{e}^{\frac{2 \gamma}{(2 \gamma_\Gamma)^2 \epsilon}} \sum_{j=1}^N \int_0^1 \int_0^{\infty} | \mathrm{e}^{\gamma x^2} u_j |^2 \d x \d t\\
& \lesssim R^2 \sup_{t \in [0,1]} \| \mathrm{e}^{\gamma x^2} \bu(t) \|^2_{2}.\end{aligned}
\end{equation}
 Taking into account that the integrand in the last term of \eqref{3terms} is supported now in $x \in [M,2M]$ and $t \in \bigg[ \dfrac{1}{2R}, 1- \dfrac{1}{2R} \bigg]$, similarly as before we get
\begin{equation}
\begin{aligned}
\varphi^k_j(x,t) &\leq (\alpha^k_j)^2 \mu (1+\epsilon) x^2 + R^2 t^2 (1-t)^2 \mu \bigg(1+\dfrac{1}{\epsilon}\bigg) \\
& \leq \gamma x^2 + \dfrac{R^2}{16} \mu \bigg( 1 + \dfrac{1}{\epsilon} \bigg)  \leq \gamma x^2 + \dfrac{R^2 \gamma}{16 (2 \gamma_\Gamma)^2 \epsilon}.
\end{aligned}
\end{equation}
Therefore,
\begin{align*}
\sum _{k=1}^N \sum_{j=1}^N &\int_0^1 \int_0^{\infty} | \mathrm{e}^{\varphi^k_j} \eta_R (\theta''_M u_j + 2 \theta'_M u_{j,x}) |^2  \d x \d t \\
 &\leq \dfrac{N}{M^2} \mathrm{e}^{\frac{2 R^2 \gamma}{16 (2 \gamma_\Gamma)^2 \epsilon}} \sum_{j=1}^N \int_{\frac{1}{2R}}^{1-\frac{1}{2R}} \int_0^{\infty}  | \mathrm{e}^{\gamma x^2}(u_j+u_{j,x}) |^2 \d x \d t\\
& \leq \dfrac{C_1}{M^2}  \mathrm{e}^{C_2 R^2} \bigg[ \sum_{j=1}^N \int_{\frac{1}{2R}}^{1-\frac{1}{2R}} \int_0^{\infty} | \mathrm{e}^{\gamma x^2} u_j |^2 \d x\d t +  \sum_{j=1}^N \int_{\frac{1}{2R}}^{1-\frac{1}{2R}} \int_0^{\infty} | \mathrm{e}^{\gamma x^2} u_{j,x} |^2 \d x\d t \bigg]\\
& \leq \dfrac{C_1}{M^2}  \mathrm{e}^{C_2 R^2} \bigg[ \sup_{t \in [0,1]} \| \mathrm{e}^{\gamma x^2} \bu(t)\|_2^2 + R^2 \int_{\frac{1}{2R}}^{1-\frac{1}{2R}} \int_0^{\infty} t (1-t) | \mathrm{e}^{\gamma x^2} u_{j,x} |^2 \d x\d t \bigg].
\end{align*}
Hence, thanks to Theorem \ref{thm56} the last term on \eqref{3terms} is bounded by 
\begin{equation}
\sum _{k=1}^N\sum_{j=1}^N \int_0^1 \int_0^{\infty} | \mathrm{e}^{\varphi^k_j} \eta_R(\theta''_M u_j + 2 \theta'_M u_{j,x}) |^2 \d x\d t \leq \dfrac{C_1}{M^2} R^2 \mathrm{e}^{C_2 R^2} \mathcal{K}.
\end{equation}
 Gathering now all these estimates, we have that
\begin{equation}
\dfrac{R^2 \epsilon}{8 \mu} \sum _{k=1}^N \sum_{j=1}^N \| \mathrm{e}^{\varphi^k_j}q_j \|_{L^2([0,1]\times \Gamma)} ^2 \leq C R^2 + \dfrac{C_1 R^2}{M^2} \mathrm{e}^{C_2 R^2} \mathcal{K.}
\end{equation}
 On the other hand for each $j$ we can find a $k$ such that $\alpha_j^k=-1$. By discarding all the other values of $k$, 
\begin{equation}
\sum _{k=1}^N \sum_{j=1}^N \| \mathrm{e}^{\varphi^k_j} q_j \|_{L^2([0,1]\times \Gamma)} ^2 \geq \sum_{j=1}^N \| \mathrm{e}^{\mu| - x + R t (1-t) |^2 - (1+\epsilon) \frac{R^2 t (1-t}{16 \mu}}  q_j\|_{L^2([0,1]\times \Gamma)} ^2.
\end{equation}
If $x \in \bigg[0, \dfrac{\epsilon (1-\epsilon)^2 R}{4} \bigg]$ and $t \in \bigg[ \dfrac{1-\epsilon}{2}, \dfrac{1+\epsilon}{2} \bigg]$, then
$
\theta_M=1$ for  $M \gg R$ and 
$
\eta_R=1$ for  $ {1}/{R} < (1-\epsilon)/{2}.
$
 Moreover, in this region, $|-x+Rt(1-t)|>Rt(1-t)-x > R(1-\epsilon)^3/4$, so
\begin{equation}
\mu (- x + R t (1-t))^2 - \dfrac{(1+\epsilon) R^2 t (1-t)}{16 \mu} \geq \dfrac{1}{4} \dfrac{R^2}{ 16 \mu} (4 \mu^2 (1-\epsilon)^6 -(1+\epsilon)^3)>0,
\end{equation}
 since $\mu > \dfrac{(1+\epsilon)^{3/2}}{2(1-\epsilon)^3}$. Hence, there exists a constant $C_{\gamma,\epsilon}$ such that
\begin{equation}
\sum_{j=1}^N \| \mathrm{e}^{\mu (-1 x + R t (1-t))^2 - \frac{(1+\epsilon) R^2 t (1-t)}{16 \mu}} q_j\|_{L^2([0,1]\times \Gamma)} ^2 \geq \mathrm{e}^{C_{\gamma,\epsilon} R^2} \sum_{j=1}^N \|  u_j \|^2_{L^2\big([\frac{1-\epsilon}{2}, \frac{1+\epsilon}{2}] \times  [0,{\frac{\epsilon(1-\epsilon)^2 R}{4}}] \big)}.
\end{equation}
Thus we show that there exists a positive constant $\mathcal{C}_{\gamma,\epsilon, \bv}$ such that
\begin{equation}
\mathrm{e}^{C_{\gamma,\epsilon}R^2} \sum_{j=1}^N \| u_j \|_{L^2\big([\frac{1-\epsilon}{2}, \frac{1+\epsilon}{2}] \times  [0,{\frac{\epsilon(1-\epsilon)^2 R}{4}}] \big)} \leq \mathcal{C}_{\gamma,\epsilon, \bv} +  \mathcal{C}_{\gamma,\epsilon, \bv} \dfrac{\mathrm{e}^{C_2 R^2}}{M^2}  \mathcal{K} .
\end{equation}
By letting $M$ tend to infinity, we have
\begin{equation}
\mathrm{e}^{C_{\gamma,\epsilon}R^2} \sum_{j=1}^N \| u_j \|_{L^2\big([\frac{1-\epsilon}{2}, \frac{1+\epsilon}{2}] \times  [0,{\frac{\epsilon(1-\epsilon)^2 R}{4}}] \big)} \leq \mathcal{C}_{\gamma,\epsilon, \bv} .
\end{equation}
Next, using the error estimate
\begin{equation}
\mathcal{N}^{-1} \|  u_j(0) \|_2 \leq \| u_j(t) \|_2 \leq \mathcal{N} \| u_j(0) \|_2,\ \ \mathcal{N}=\mathrm{e}^{\sup_{[0,1]}\|\Im \bv(t)\|_\infty},
\end{equation}
which can be proved in the same way as the analogous estimate in the real line, and
\begin{equation}
\| u_j(t) \|_2 \leq \| u_j(t) \|_{L^2\big(  [0, {\frac{\epsilon(1-\epsilon)^2 R}{4}}] \big)} + \mathrm{e}^{-\gamma R^2 \frac{\epsilon^2 (1-\epsilon)^4}{16}} \mathcal{C}_{\gamma,\epsilon, \bv}, \quad 0 \leq t \leq 1,
\end{equation}
 we show that there exists a positive constant
\begin{equation}
\widetilde{C}_{\gamma,\epsilon}=\min \bigg\{ C_{\gamma,\epsilon}, \dfrac{\gamma \epsilon^2(1-\epsilon)^4}{16} \bigg\},
\end{equation}
such that $\mathrm{e}^{\widetilde{C}_{\gamma,\epsilon} R^2} \| \bu(0) \|_{2} \leq \mathcal{C_{\gamma,\epsilon,\bv}}.$
Indeed, 
\begin{align}
\mathrm{e}^{\widetilde{C}_{\gamma,\epsilon} R^2} \| \bu(0) \|_{2} & \leq  \mathcal{N}\mathrm{e}^{\widetilde{C}_{\gamma,\epsilon} R^2}\sum_{j=1}^N \| u_j(t) \|_{2} \\ 
&\lesssim \mathcal{N} \mathrm{e}^{\widetilde{C}_{\gamma,\epsilon}R^2} \sum_{j=1}^N \| u_j \|_{L^2\big([0, {\frac{\epsilon(1-\epsilon)^2 R}{4}}] \big)}+  \mathrm{e}^{\widetilde{C}_{\gamma,\epsilon}R^2 -\gamma R^2 \frac{\epsilon^2 (1-\epsilon)^4}{16}}
\end{align}
and hence, integrating the last inequality in $t\in[{\frac{1-\epsilon}{2}}, {\frac{1+\epsilon}{2}}]$,
\begin{equation}
\begin{aligned}
C_{\epsilon} \mathrm{e}^{\widetilde{C}_{\gamma,\epsilon} R^2} \| \bu(0) \|_{2} & \leq \mathcal{N} \mathrm{e}^{\widetilde{C}_{\gamma,\epsilon}R^2} \sum_{j=1}^N \| u_j \|_{L^2\big( [{\frac{1-\epsilon}{2}}, {\frac{1+\epsilon}{2}}]  \times [0, {\frac{\epsilon(1-\epsilon)^2 R}{4}}] \big)}+ C_{\epsilon} \mathrm{e}^{\widetilde{C}_{\gamma,\epsilon} -\gamma \frac{\epsilon^2 (1-\epsilon)^4}{16} R^2} \leq \mathcal{C}.
\end{aligned}
\end{equation}
Letting $R$ go to infinity, we conclude $\bu \equiv 0$.
\end{proof}

\section{Appendix. Appell transform}\label{sec-Appell transform}

In order to study the behavior of solutions of the Schr\"odinger equation with a potential, we will restrict ourselves to the case where the rates of decay at times $t=0$ and $t=1$ are the same. We will reduce from the general case to this case by means of the so-called Appell transformation (see \cite{MR2443923} for the proof).


\begin{lemma}
Assume that $\bu(s,y)$ verifies
\[
\partial_s\bu =(A+iB)(\Delta_\Gamma \bu+\bv(s,y)\bu+{\bf F}(s,y)),\ \text{ in } [0,1]\times\Gamma,
\]
where $A+iB\ne0,\ \alpha$ and $\beta$ are positive, $\gamma\in \mathbb{R}$, and set
\[
\tilde{\bu}(t,x)=\left(\frac{(\alpha\beta)^{1/4}}{\sqrt{\alpha}(1-t)+\sqrt{\beta} t}\right)^{1/2}\bu\left(\frac{\sqrt{\beta} t}{\sqrt{\alpha}(1-t)+\sqrt{\beta} t},\dfrac{(\alpha\beta)^{1/4}x}{\sqrt{\alpha}(1-t)+\sqrt{\beta} t}\right)\mathrm{e}^{\frac{(\sqrt{\alpha}-\sqrt{\beta})|x|^2}{4(A+iB)(\sqrt{\alpha}(1-t)+\sqrt{\beta} t)}}.
\]

Then $\tilde{\bu}$ verifies
\[
\partial_t\tilde{\bu} =(A+iB)(\Delta \tilde{\bu}+\tilde{\bv}(t,x)\tilde{\bu}+\tilde{\bf F}(t,x)),\ \text{ in } [0,1]\times\Gamma,
\]
with
\[
\tilde{\bv}(t,x)=\frac{\sqrt{\alpha\beta}}{(\sqrt{\alpha}(1-t)+\sqrt{\beta} t)^2}\bv\left(\frac{\sqrt{\beta} t}{\sqrt{\alpha}(1-t)+\sqrt{\beta} t},\dfrac{(\alpha\beta)^{1/4}x}{\sqrt{\alpha}(1-t)+\sqrt{\beta} t}\right),
\]
\[
\tilde{\bf F}(t,x)=\left(\frac{(\alpha\beta)^{1/4}}{\sqrt{\alpha}(1-t)+\sqrt{\beta} t}\right)^{5/2}{\bf F}\left(\frac{\sqrt{\beta} t}{\sqrt{\alpha}(1-t)+\sqrt{\beta} t},\dfrac{(\alpha\beta)^{1/4}x}{\sqrt{\alpha}(1-t)+\sqrt{\beta} t}\right)\mathrm{e}^{\frac{(\sqrt{\alpha}-\sqrt{\beta})|x|^2}{4(A+iB)(\sqrt{\alpha}(1-t)+\sqrt{\beta} t)}}.
\]

Moreover
\[
\|\mathrm{e}^{\gamma |x|^2}\tilde{\bf F}(t)\|=\frac{\sqrt{\alpha\beta}}{(\sqrt{\alpha}(1-t)+\sqrt{\beta} t)^2}\|\mathrm{e}^{\left[\frac{\gamma \sqrt{\alpha\beta}}{(\sqrt{\alpha} s+\sqrt{\beta} (1-s))^2}+\frac{(\sqrt{\alpha}-\sqrt{\beta})A}{4(A^2+B^2)(\sqrt{\alpha} s+\sqrt{\beta} (1-s))}\right] |y|^2}{\bf F}(s)\|
\]
and
\[
\|\mathrm{e}^{\gamma |x|^2}\tilde{\bu}(t)\|=\|\mathrm{e}^{\left[\frac{\gamma \sqrt{\alpha\beta}}{(\sqrt{\alpha} s+\sqrt{\beta} (1-s))^2}+\frac{(\sqrt{\alpha}-\sqrt{\beta})A}{4(A^2+B^2)(\sqrt{\alpha} s+\sqrt{\beta} (1-s))}\right]|y|^2}{\bu}(s)\|
\]
when $s=\frac{\sqrt{\beta} t}{\sqrt{\alpha}(1-t)+\sqrt{\beta} t}$.
\end{lemma}

The proof is based on explicit computations and the fact that the first derivative of function $x\to \exp(-x^2)$ vanishes at $x=0$.

\bibliographystyle{plain}
\bibliography{biblio}

\begin{thebibliography}{10}

\bibitem{MR2804557}
R~Adami, C~Cacciapuoti, D~Finco, and D~Noja.
\newblock {Fast solitons on star graphs}.
\newblock {\em Rev. Math. Phys.}, 23(4):409--451, may 2011.

\bibitem{MR1476363}
C.~Cattaneo.
\newblock The spectrum of the continuous {L}aplacian on a graph.
\newblock {\em Monatsh. Math.}, 124(3):215--235, 1997.

\bibitem{cowlingekpv2010}
M.~Cowling, L.~Escauriaza, C.~E. Kenig, G.~Ponce, and L.~Vega.
\newblock The {H}ardy uncertainty principle revisited.
\newblock {\em Indiana Univ. Math. J.}, 59(6):2007--2025, 2010.

\bibitem{cowlingprice}
Michael Cowling and John~F. Price.
\newblock Generalisations of {H}eisenberg's inequality.
\newblock In {\em Harmonic analysis ({C}ortona, 1982)}, volume 992 of {\em
  Lecture Notes in Math.}, pages 443--449. Springer, Berlin, 1983.

\bibitem{MR2443923}
L.~Escauriaza, C.~E. Kenig, G.~Ponce, and L.~Vega.
\newblock Hardy's uncertainty principle, convexity and schr\"odinger
  evolutions.
\newblock {\em J. Eur. Math. Soc. (JEMS)}, 10(4):883--907, 2008.

\bibitem{ekpvsharp}
Luis Escauriaza, Carlos~E. Kenig, Gustavo Ponce, and Luis Vega.
\newblock The sharp {H}ardy uncertainty principle for {S}chr\"{o}dinger
  evolutions.
\newblock {\em Duke Math. J.}, 155(1):163--187, 2010.

\bibitem{gaveauokada87}
Bernard Gaveau, Masami Okada, and Tatsuya Okada.
\newblock Second order differential operators and {D}irichlet integrals with
  singular coefficients. {I}. {F}unctional calculus of one-dimensional
  operators.
\newblock {\em Tohoku Math. J. (2)}, 39(4):465--504, 1987.

\bibitem{hardyup}
G.~H. Hardy.
\newblock A {T}heorem {C}oncerning {F}ourier {T}ransforms.
\newblock {\em J. London Math. Soc.}, 8(3):227--231, 1933.

\bibitem{ignatsiam}
Liviu~I. Ignat.
\newblock Strichartz estimates for the {S}chr\"{o}dinger equation on a tree and
  applications.
\newblock {\em SIAM J. Math. Anal.}, 42(5):2041--2057, 2010.

\bibitem{ignatpazotorosier}
Liviu~I Ignat, Ademir~F Pazoto, and Lionel Rosier.
\newblock Inverse problem for the heat equation and the schr{\"o}dinger
  equation on a tree.
\newblock {\em Inverse Problems}, 28(1):015011, 2011.

\bibitem{MR2277618}
Vadim Kostrykin and Robert Schrader.
\newblock Laplacians on metric graphs: eigenvalues, resolvents and semigroups.
\newblock In {\em Quantum graphs and their applications}, volume 415 of {\em
  Contemp. Math.}, pages 201--225. Amer. Math. Soc., Providence, RI, 2006.

\bibitem{MR2459876}
P.~Kuchment.
\newblock Quantum graphs: an introduction and a brief survey.
\newblock In {\em Analysis on graphs and its applications}, volume~77 of {\em
  Proc. Sympos. Pure Math.}, pages 291--312. Amer. Math. Soc., Providence, RI,
  2008.

\bibitem{morganup}
G.~W. Morgan.
\newblock A {N}ote on {F}ourier {T}ransforms.
\newblock {\em J. London Math. Soc.}, 9(3):187--192, 1934.

\bibitem{nazarov}
F.~L. Nazarov.
\newblock Local estimates for exponential polynomials and their applications to
  inequalities of the uncertainty principle type.
\newblock {\em Algebra i Analiz}, 5(4):3--66, 1993.

\bibitem{rudin}
Walter Rudin.
\newblock {\em Real and complex analysis}.
\newblock McGraw-Hill Book Co., New York, third edition, 1987.

\end{thebibliography}
\end{document}